\documentclass[11pt]{amsart}
\usepackage{tikz}
\usepackage{tikz-cd}
\usepackage{mathrsfs}
\usepackage[mathcal]{euscript}
\usepackage{graphicx}
\usepackage{amsthm,amscd}
\usepackage{calc}
\usepackage{color}

\usepackage{mathrsfs,dsfont}
\usepackage{fancyhdr,amscd}
\usepackage{anysize}
\usepackage{amsmath,amssymb,amsfonts}
\usepackage{epsfig,enumerate}
\usepackage{latexsym}
\usepackage{indentfirst, latexsym}
\usepackage{graphics}
\usepackage[all,poly,knot]{xy}
\usepackage[pagebackref]{hyperref}
\usepackage{lineno}

\newcommand{\comment}[1]{}

\usepackage{fancyhdr}
\providecommand{\U}[1]{\protect\rule{.1in}{.1in}}

\numberwithin{equation}{section}

\theoremstyle{plain}
\newtheorem{theorem}{Theorem}[section]  
\newtheorem{lemma}[theorem]{Lemma}      

\newtheorem{proposition}[theorem]{Proposition}
\newtheorem{corollary}[theorem]{Corollary}

\theoremstyle{definition}
\newtheorem{definition}[theorem]{Definition}

\newtheorem{remark}[theorem]{Remark}
\newtheorem{claim}[theorem]{Claim}
\newtheorem{observation}[theorem]{Observation}
\newtheorem{fact}[theorem]{Fact}

\newcounter{proofstep}
\newenvironment{step}[1][]{%
	\stepcounter{proofstep}%
	\par\vspace{12pt}%
	\noindent\textbf{Step \theproofstep. #1}\quad
}{\par\vspace{12pt}}

\begin{document}
	
	\title[BBF positivity and   Moishezonness]{Beauville--Bogomolov--Fujiki positivity and Moishezonness of
		complex symplectic manifolds in Fujiki class $\mathscr C$}
	
	\author[Jian Chen]{Jian Chen}
	
	\address{Jian Chen, School of Mathematics and Statistics, Central China Normal University,
		Wuhan 430079, People's Republic of China}
	\email{jian-chen@whu.edu.cn}

	\date{\today}

	\begin{abstract}
		Motivated by the principle that positivity of the Beauville--Bogomolov--Fujiki (BBF) form controls the geometry of hyperk\"ahler manifolds, we study BBF positivity from the perspective of bimeromorphic geometry, focusing on its relation  to the Moishezonness of hyperfujiki manifolds
		(natural bimeromorphic analogues of hyperk\"ahler manifolds).
		We prove  that any Moishezon hyperfujiki manifold
		carries a big line bundle with positive BBF square. We also prove that
		the existence of a BBF-positive integral $(1,1)$-class implies
		Moishezonness for hyperfujiki $4$-folds and, more generally, for
		hyperfujiki $2n$-folds under a weak K\"ahler minimal-model condition.
		The proof mainly uses the theory of primitive symplectic
		varieties developed by B. Bakker--C. Lehn,  a  
		construction  of small bimeromorphic models for certain $K$-trivial
		manifolds by I. Biswas--J. Cao--S. Dumitrescu--H. Guenancia, and the theory of pull-backs of reflexive differential forms established by
		S. Kebekus--C. Schnell.  As an application, we give a
		Hodge-theoretic description of the Moishezon locus in certain smooth
		families by combining this criterion with the period theory developed
		by B. Anthes--A. Cattaneo--S. Rollenske--A. Tomassini.

	\end{abstract}
	
	\maketitle
	
	\setcounter{tocdepth}{1}
	\tableofcontents

	\section{Introduction}
	
	A basic principle in hyperk\"ahler geometry (e.g., \cite[Introduction]{Huy99}) is that, for a 
	hyperk\"ahler manifold $X$, the Beauville--Bogomolov--Fujiki (BBF) form (Definitions \ref{def-bbf-fourfold} and \ref{def-bbf-primitive-fourfold}) $q_X$ on 
	$H^2(X,\mathbb{Z})$ controls the geometry of $X$. In this paper, we are only concerned with the geometry described by BBF positivity.  For example, Huybrechts showed that 
	a hyperk\"ahler manifold 
	$X$ is projective if and only if there exists a class
	$\alpha\in H^2(X,\mathbb Z)\cap H^{1,1}(X)$ such that
	$q_X(\alpha)>0$ (\cite[Theorem 2]{Huy01}, \cite[Theorem 3.11]{Huy99}). 
	Recently, Bakker--Lehn generalized (\cite[Theorem 1.2]{BL22}) Huybrechts' projectivity criterion to primitive
	symplectic varieties (Definition \ref{def-symp-varie}).
	
	Note that such BBF-theoretic projectivity criteria have important applications, such as showing
	that birational projective hyperk\"ahler manifolds are deformation equivalent and
	diffeomorphic
	(\cite[Theorem 4.6 and Corollary 4.7]{Huy99})
	and that any primitive symplectic variety is locally trivially deformation equivalent
	to a projective primitive symplectic variety
	(\cite[Corollary 1.3]{BL22}).

	In this paper, we investigate the corresponding BBF-positivity questions
	from the viewpoint of bimeromorphic geometry. In this context, Moishezon
	manifolds are the natural bimeromorphic counterparts of projective
	manifolds, while hyperfujiki manifolds
	(Definition \ref{def-hyperfujiki-hyperka}) play the corresponding role of
	hyperk\"ahler manifolds (as shown by
	Remark \ref{remark-hyfu-nothyperkah}).  Note that hyperfujiki manifolds form  a special class of simple
	$\partial\bar\partial$-complex symplectic manifolds
	(Definition \ref{sec-prelim-ddbar}). Thanks to the work of Anthes--Cattaneo--Rollenske--Tomassini
	(\cite{ACRT18}) and Cattaneo--Tomassini (\cite{CT18}), the theory of BBF
	forms and period maps for such manifolds has been developed.

	A natural question arises: is a hyperfujiki manifold Moishezon if and only if it carries an integral $(1,1)$-class with positive BBF square? In the only-if direction, both \cite{Huy01} and \cite{BL22} use the first Chern class of an ample line bundle, which is a K\"ahler class and therefore has positive BBF square. For the converse direction, \cite{Huy01} needs to approximate a BBF-positive class by K\"ahler classes, whereas \cite{BL22} replaces K\"ahler classes with Demailly--P\u{a}un classes. However, a non-K\"ahler Moishezon manifold has no K\"ahler classes and may well have no Demailly--P\u{a}un classes (\cite[Remark 6.3]{BL22}). Thus, the arguments of \cite{Huy01} and \cite{BL22} do not seem directly adaptable to Moishezon hyperfujiki manifolds. Recall that a compact complex manifold is Moishezon if and only if it admits a big line bundle. However, even on a projective hyperk\"ahler manifold, a big line bundle may have negative BBF square (see Fact \ref{fact-big-notposi}). Therefore, the only-if direction does not follow directly from the existence of a big line bundle, a priori.

	Our rough strategy is as follows. By combining I. Biswas–J. Cao–
	S. Dumitrescu–H. Guenancia \cite[Proof of Theorem 3.8, Step 1]{BCDG25}  with Claudon--H\"oring \cite[Lemma 2.11]{CH24}, we obtain, for any Moishezon hyperfujiki manifold $Y$,  
	a compact K\"ahler strongly $\mathbb Q$-factorial primitive
	symplectic variety $X$ with only terminal singularities, together with
	a small bimeromorphic map
	\[
	\phi:Y\dashrightarrow X.
	\]
	Note that $X$ is  projective by  Namikawa's
	projectivity criterion \cite[Theorem 1.6]{Nam02}. Thus, we may choose an ample line
	bundle $L_X$ on $X$, whose first Chern class has positive BBF square.
	We then transport $L_X$ to a big line bundle $L_Y$ on $Y$ and compare
	$q_X(c_1(L_X))$ with $q_Y(c_1(L_Y))$. We show that these two quantities  agree up to multiplication by a positive constant, which yields  $q_Y(c_1(L_Y))>0.$

	This strategy leads to the following theorem, which asserts that any
	Moishezon hyperfujiki manifold admits a big line bundle whose first
	Chern class has positive BBF square.

	\begin{theorem}[{=Theorem \ref{thm-arb-positive}}]\label{thm-arb-positive-intro}
		Let $(Y,\sigma_Y)$ be a hyperfujiki $2n$-fold and $q_Y$ the BBF form. Then
		\[
		Y\text{ is Moishezon}
		\Longrightarrow \text{ there exists a big line bundle } L_Y \text{ such that }
		q_Y(c_1(L_Y))>0.
		\]
	\end{theorem}

	We briefly describe the main technical point in the proof of Theorem \ref{thm-arb-positive-intro}. Once the
	small bimeromorphic map $\phi:Y\dashrightarrow X$ and the ample line
	bundle $L_X$ have been obtained, the construction of $L_Y$ is standard:
	we pass to a common resolution
	\[
	\mu:W\longrightarrow Y,
	\qquad
	\nu:W\longrightarrow X,
	\]
	and take the reflexive hull on $Y$ of the corresponding sheaf. This yields the required big line bundle $L_Y$.
	The main difficulty is to compare
	$q_X\bigl(c_1(L_X)\bigr)$ with $q_Y\bigl(c_1(L_Y)\bigr)$.
	
	Philosophically motivated by  the   coisotropicity-type (\cite[Theorem 1.2]{Wie03}) and isotropicity-type (\cite[Lemma 2.1]{WW03} or  \cite[Chapter 4]{Wie00} or \cite[\S 4]{Kl01}) phenomena for certain exceptional divisors, for obtaining this comparison, we establish the vanishing (see Claim
	\ref{claim-boundary-power-vanishing}) of the $(n-1)$-power of the pulled-back symplectic form along certain boundary divisors.
	In proving this boundary
	vanishing,  the existence, functoriality, and compatibility of
	pull-backs of reflexive differential forms established by
	Kebekus--Schnell
	\cite[Theorems 1.10 and 14.1, Note 1.6.1, and Fact 14.3]{KS21}
	play an important role: they allow the relevant
	pull-backs to be compared even when the image under consideration is
	contained entirely in the singular locus of $X$.

	For the converse direction, motivated by the $4$-dimensional K\"ahler
	minimal model program of Das--Hacon--P\u{a}un \cite{DHP24}, we introduce the
	condition $\left(\mathrm{KMM}_{0,2n}\right)$
	(Definition \ref{def-kmm-zero}). This condition is weaker than assuming
	the full K\"ahler minimal model program in dimension $2n$ (see Observation \ref{rem-kmm-relation-dim 4}). Under
	$\left(\mathrm{KMM}_{0,2n}\right)$, we can likewise construct a small
	primitive symplectic model $X$ for any hyperfujiki $2n$-fold $Y$.
	If $Y$ carries a holomorphic line bundle $L_Y$ such that
	$q_Y(c_1(L_Y))>0$, we may produce a
	corresponding holomorphic line bundle $L_X$ on $X$.
	The aforementioned comparison
	remains valid and yields
	$q_X(c_1(L_X))>0$. The projectivity criterion of Bakker--Lehn
	\cite[Theorem 1.2]{BL22} then implies that $X$ is projective. Thus $Y$
	is Moishezon.

	\begin{theorem}[{=Corollaries \ref{cor-hf-moishezon}+\ref{cor-hf-moishezon-4dim}}]
		\label{cor-hf-moishezon-intro}
		Let $(Y,\sigma_Y)$ be a hyperfujiki manifold. Assume that either
		$\dim Y=4$, or $\dim Y=2n$ and
		$\left(\mathrm{KMM}_{0,2n}\right)$ holds ($n\geq 3$).
		Then
		\[
		Y\text{ is Moishezon}
		\Longleftrightarrow
		\exists v\in H^2(Y,\mathbb{Z})\cap H^{1,1}(Y)
		\text{ such that }q_Y(v)>0.
		\]
	\end{theorem}

	As an application of the BBF-positivity characterization of Moishezonness (and projectiveness), combined with the period theory developed by Anthes–-Cattaneo–-Rollenske--Tomassini \cite{ACRT18}, we obtain the following description of the Moishezon locus for certain smooth families.

	\begin{theorem}[{=Theorem \ref{thm-intro-hf-copy2}}]\label{thm-intro-hf-copy2-intro}
		Let $f:\mathcal{X}\to S$ be a proper holomorphic submersion from a complex
		manifold $\mathcal{X}$ to a connected simply connected complex manifold $S$, with
		connected fibers. Assume that one of the following two conditions holds:
		\begin{enumerate}
			\item 
			either each fiber $X_t:=f^{-1}(t)$ is a hyperfujiki $4$-fold, or each fiber $X_t$ is a hyperfujiki $2n$-fold and $\left(\mathrm{KMM}_{0,2n}\right)$ holds ($n\geq 3$);
			\item every fiber $X_t$ is a hyperk\"ahler manifold.
		\end{enumerate}
		Then the Moishezon locus, which in the hyperk\"ahler case coincides with the
		projective locus,
		\[
		\operatorname{Moi}(f):=\{\,t\in S\mid X_t\text{ is Moishezon}\,\},
		\]
		is either  $S$ or    at most  (possibly empty) a countable
		union of hypersurfaces of $S$.
	\end{theorem}

	The remaining sections are organized as follows.
	In Section \ref{sec-preli}, we recall the necessary preliminaries on the
	Fujiki class, Moishezon manifolds, $\partial\bar\partial$-complex symplectic
	manifolds, hyperfujiki manifolds, primitive symplectic varieties, and BBF
	forms.
	In Section \ref{sec-arb-dim-moishezon-positive}, we construct the small
	primitive symplectic models needed for the argument and prove
	Theorem \ref{thm-arb-positive}; the main ingredients are the boundary-power
	vanishing in Claim \ref{claim-boundary-power-vanishing} and the resulting
	comparison of the BBF forms.
	In Section \ref{sec-posi-mois}, we prove Theorem
	\ref{thm-positive-moishezon}, introduce the weak K\"ahler minimal-model
	condition $\left(\mathrm{KMM}_{0,d}\right)$, verify
	$\left(\mathrm{KMM}_{0,4}\right)$, and derive the BBF criterion for the
	Moishezonness of hyperfujiki $4$-folds and, under
	$\left(\mathrm{KMM}_{0,2n}\right)$, of hyperfujiki $2n$-folds.
	Finally, in Section \ref{sec-main-hf}, we combine these criteria with period
	theory to describe the Moishezon locus in certain smooth families in terms
	of the Hodge loci of BBF-positive integral classes.

	\section{Preliminaries}\label{sec-preli}
	
	Unless otherwise stated, throughout this paper: we always work in the complex analytic category;  we use the notation $q_{\bullet}$ to denote the BBF (Beauville--Bogomolov--Fujiki) form of  $\bullet$ when it is well-defined (e.g., Definition \ref{def-bbf-fourfold}); When we use the term ``nef", we always mean nefness in the analytic sense (e.g., \cite[Definition 2.1-(2)]{DHP24}); we use $\omega_{\bullet}$ and $K_{\bullet}$ interchangeably to denote the canonical sheaf of a normal variety $\bullet$; By a \emph{ complex analytic variety (or analytic variety, or simply variety)}, we mean a reduced and irreducible complex analytic space;
	If $V$ is a normal variety and
	$\mathscr{F}$ is a  coherent sheaf on $V$, \emph{$\mathscr{F}^{[m]}$} is defined to be  $\bigl(\mathscr{F}^{\otimes m}\big)^{**}$, where $(\bullet)^{**}$ is the double dual of $\bullet$; By a \emph{resolution} of a variety $W$, we mean a proper modification (e.g., \cite[\S 2]{Ue75}) from a complex manifold to $W$.

	\subsection{The Fujiki class and Moishezon manifolds}

	\begin{definition}[{\cite[\S 4.3]{Fu78/79}, \cite[Chapter IV \S 3]{Vr89}}]\label{def-fujiki-c}
		Let $X$ be a reduced compact complex analytic space. We say that $X$ is \emph{in Fujiki class $\mathscr{C}$}, and write $X\in\mathscr{C}$, if there exist a reduced compact K\"ahler space $\widetilde{X}$ and a surjective meromorphic map
		$\widetilde{X}\dashrightarrow X$.  Note that  $X\in\mathscr{C}$ iff $X$ is bimeromorphic to a compact K\"ahler manifold. We will use the terms \emph{Fujiki variety} and \emph{space in Fujiki class \(\mathscr C\)}
		interchangeably; when \(X\) is smooth, we will also use the term \emph{Fujiki manifold}.
	\end{definition}

	In this paper, we use the terms in Fujiki class $\mathscr{C}$ and Fujiki variety interchangeably to denote a variety belonging to Fujiki class $\mathscr{C}$.
	
	\begin{definition}
		A compact connected complex manifold $X$ is called  \emph{a Moishezon manifold} if it possesses $\operatorname{dim}_{\mathbb{C}} X$ algebraically independent meromorphic functions. Equivalently, $X$ is Moishezon if and only if there exists a projective  manifold $Y$ that is bimeromorphic to $X$. 
	\end{definition}

	\subsection{$\mathbb Q$-factoriality and terminal singularities in the complex analytic setting}\label{sec-prelim-singu}
	
	We refer the reader to \cite[\S 1]{St88}, \cite[\S 3]{Fj22} and \cite[Definition 2.7]{DHP24}   for explanations of the related notions in the complex analytic setting, and to \cite[\S  2.1, \S 2.3, \S 5.1, \S 5.2]{KM98}, \cite[Chapter 4, \S 4.1]{Mt02}  and \cite[\S 2.1]{KK13} for systematic treatments in the algebraic setting (also carry over to  complex analytic spaces). One key difference between the algebraic and analytic settings is that, in the analytic setting, a canonical divisor need not exist.

	\begin{definition}[{e.g., \cite[Definition 2.7]{DHP24}}]
		Let  $X$ be a normal variety.     We say $X$ is \emph{$\mathbb{Q}$-factorial}, if every prime Weil divisor $D$ on $X$ is $\mathbb{Q}$-Cartier and there is a positive integer $m>0$ such that the \emph{(reflexive) $m$-canonical sheaf $\omega_X^{[m]}:=\left(\omega_X^{\otimes m}\right)^{* *}$} is invertible. 
	\end{definition}

	\begin{definition}[{e.g.,  \cite[Definition 2.8]{CH24}}]\label{def-strong-Q-fac}
		Let  $V$ be a normal variety.     We say that $V$ is \emph{strongly $\mathbb{Q}$-factorial},  if for every rank  $1$ reflexive
		sheaf $\mathscr{F}$ on $V$, there exists $m>0$ such that
		$\mathscr{F}^{[m]}$ is locally free.  Note that in \cite[Definition 2.14]{BL22},  the notion of strongly $\mathbb{Q}$-factoriality is called $\mathbb{Q}$-factoriality. 
	\end{definition}
	
	We now give a formulation of terminal singularities equivalent to \cite[Definition 1]{St88}, via discrepancies.
	
	\begin{definition}\label{def-terminal-analytic}
		Let $X$ be a normal variety. 
		We say that \emph{$X$ has only terminal singularities} if the following conditions are
		satisfied:
		\begin{enumerate}[\rm{(}1\rm{)}]
			\item There exists an integer $m>0$ such that  $\omega_X^{[m]}$
			is an invertible sheaf.
			
			\item There exists a resolution $\pi:Y\to X$
			such that the following isomorphism holds:
			\[
			\left(\omega_Y\right)^{\otimes m}
			\cong
			\pi^*\omega_X^{[m]}
			\otimes
			\mathcal O_Y\left(\sum_i ma_iE_i\right),
			\]
			where the $E_i$ is the prime divisor   contained in the exceptional
			locus   of $\pi$ and   $ a_i>0$ and  $ma_i\in\mathbb Z$ for each $i$. Note that, as in 
			\cite[Definition 2.4]{KK13} (also carry over to complex analytic spaces, as pointed out in \cite[\S 2.1]{KK13}), if $E\subseteq Y$ denotes the exceptional
			locus of $\pi$, then the natural isomorphism (induced by the biholomorphism
			$\pi|_{Y\setminus E}:Y\setminus E\to X\setminus\pi(E)$)
			\[
			\left.\left(\omega_Y\right)^{\otimes m}\right|_{Y\setminus E}
			\cong
			\left.\pi^*\omega_X^{[m]}\right|_{Y\setminus E}
			\]
			is required to
			extend to the isomorphism above, where
			$\mathcal O_Y(\sum_i ma_iE_i)|_{Y\setminus E}$ is naturally identified
			with $\mathcal O_{Y\setminus E}$.
		\end{enumerate}
	\end{definition}

	As in the scheme setting, terminality is independent of the choice of resolution. Moreover, terminal singularities are rational: if a normal variety $X$ has only terminal singularities, then, for any resolution $\pi:Y\to X$, one has
	\[
	R^i\pi_*\mathcal O_Y=0
	\qquad
	\text{for each }i>0
	\]
	(and $\pi_*\mathcal O_Y=\mathcal O_X$).

	\subsection{$\partial\bar\partial$-complex  symplectic manifolds}\label{sec-prelim-ddbar}
	
	\begin{definition}[{\cite[Definition 1.1]{ACRT18}}]\label{def-ddbar-symplectic}
		A \emph{$\partial\bar\partial$-complex symplectic manifold} is a pair
		$(X,\sigma)$, where $X$ is a compact complex manifold satisfying the
		$\partial\bar\partial$-lemma and
		$\sigma\in H^0(X,\Omega_X^2)$ is a  holomorphic $2$-form (the $\partial\bar\partial$-lemma implies that $\sigma$ is automatically $d$-closed, e.g., \cite[Footnote 5]{C25})
		which is everywhere non-degenerate.
		We say that $X$ is \emph{simple} if $\sigma$ is unique up to scalars, that is,
		\[
		H^0(X,\Omega_X^2)=\mathbb C\sigma,
		\]
		or equivalently, $h^{2,0}(X)=1$.
	\end{definition}

	For $\partial\bar\partial$-complex symplectic manifolds, there is a rich theory of BBF forms and period maps. We refer the reader to \cite{ACRT18,CT18} for further developments.

	In the present paper, we mainly study a special class of simple
	$\partial\bar\partial$-complex symplectic manifolds, namely hyperfujiki
	manifolds, which may be viewed as natural bimeromorphic analogues of
	hyperk\"ahler manifolds.
	
	\begin{definition}\label{def-hyperfujiki-hyperka}
		Let $(X,\sigma_X)$ be a compact complex manifold endowed with an everywhere
		non-degenerate holomorphic two-form $\sigma_X$.   We say that $(X,\sigma_X)$ is
		\emph{hyperfujiki} (e.g., \cite[Corollary 4.11]{C25}), respectively \emph{hyperk\"ahler}, if $X$ is simply connected,
		$X$ belongs to Fujiki class $\mathscr C$, respectively $X$ is K\"ahler, and
		\[
		H^0(X,\Omega_X^2)=\mathbb C\sigma_X.
		\]
		In both cases, $X$ satisfies the $\partial\bar\partial$-lemma and thus $\sigma_X$ is automatically $d$-closed. 
		Moreover, since
		$\sigma_X$ is everywhere non-degenerate, $X$ is even-dimensional; writing
		$\dim X=2n$, the form $\sigma_X^n$ trivializes $K_X$. Thus
		$K_X\cong\mathcal O_X$.
	\end{definition}

	\begin{remark}\label{remark-hyfu-nothyperkah}
		There exist hyperfujiki manifolds which are not hyperk\"ahler.  Indeed, let
		$X$ be a  hyperk\"ahler $2n$-fold containing a Lagrangian projective space $P\cong\mathbb P^n$, where $n\geq 2$.  Then its Mukai
		flop $X^{\prime}$ is complex symplectic and belongs to Fujiki class
		$\mathscr C$, and suitable choices of $X$ yield a non-K\"ahler
		$X^{\prime}$  (\cite[\S 4.4, Proposition 4.18]{Y01}, \cite[Examples 21.7 and 21.9]{GHJ02}, \cite[Example 5.1]{ACRT18}). 
	\end{remark}
	
	Note that every integral $(1,1)$-class on a hyperfujiki manifold  $Y$ is the first Chern class of a
	holomorphic line bundle on $Y$. Indeed, the  simply connectedness and the $\partial\bar\partial$-lemma give that $H^1(Y,\mathcal{O}_Y)=H^{0,1}(Y)=0$.  Then the long exact sequence associated with 
	the exponential sequence
	gives   an isomorphism
	\begin{equation}\label{iso-pica}
		c_1:\operatorname{Pic}(Y)\xrightarrow{\cong}H^2(Y,\mathbb{Z})\cap H^{1,1}(Y).
	\end{equation}
	Furthermore, for a hyperfujiki 
	manifold $X$, the simply connectedness and the universal coefficient theorem
	show that $H^2\left(X, \mathbb{Z}\right) \cong \operatorname{Hom}_{\mathbb{Z}}\left(H_2\left(X, \mathbb{Z}\right), \mathbb{Z}\right)$. Thus  $H^2\left(X, \mathbb{Z}\right)$ has no torsion element.

	\subsection{Primitive symplectic varieties and BBF forms}
	Although the main objects of this paper are hyperfujiki and hyperk\"ahler manifolds,
	the theory of (singular) symplectic varieties by  Bakker--Lehn \cite{BL22}, plays an essential role in our arguments. We now
	introduce several notions that will be used later.
	
	\begin{definition}[{\cite[Definition 3.1]{BL22}, \cite[Definition 1.1]{Bea00}}]\label{def-symp-varie}
		A \emph{symplectic variety} is a pair $(X, \sigma)$  consisting of a  normal variety $X$ and a closed holomorphic \emph{symplectic} (i.e., non-degenerate everywhere) form $\sigma \in H^0\left(X_{\text {reg }}, \Omega_X^2\right)$ on $X_{\text {reg }}$ such that there is a resolution   of singularities $\pi: \widetilde{X} \to X$ for which $\pi^* \sigma$ extends to a holomorphic form on $\widetilde{X}$. A \emph{ primitive symplectic variety} is a normal compact K\"ahler variety $X$ such that $H^1\left(X, \mathcal{O}_X\right)=0$ and $H^0\left(X_{\text {reg }}, \Omega_X^2\right)=\mathbb{C} \sigma$ such that $(X, \sigma)$ is a symplectic variety.   Note that a symplectic variety has rational singularities (e.g., \cite[Theorem 3.4]{BL22}).
	\end{definition}

	\begin{remark}\label{rem-singular-hodge-and-type}
		Note that if a variety \(X\) has rational singularities and admits a resolution by a
		manifold in Fujiki class \(\mathscr C\), then (the torsion free part of) \(H^2(X,\mathbb Z)\)
		carries a pure weight $2$ Hodge structure (\cite[Introduction and Lemma 2.1]{BL22}).
		Moreover, if \(X\) is  normal and
		\(j:  X_{\rm reg}\hookrightarrow X\) denotes the inclusion, then $(\Omega_X^{p})^{**}=j_*\Omega^p_{X_{\rm reg}}.$
		If \(X\) is further assumed to be  in Fujiki class \(\mathscr C\), then for \(p+q\leq 2\)
		the relevant graded pieces of the Hodge filtration are identified with
		\(H^q(X,(\Omega_X^{p})^{**})\)  (\cite[Paragraph following Definition 2.2]{BL22}).  In particular, taking \(p=2\) and
		\(q=0\), we have
		\[
		H^0(X_{\rm reg},\Omega^2_{X_{\rm reg}})
		=
		H^0(X,j_*\Omega^2_{X_{\rm reg}})
		=
		H^0(X,(\Omega_X^{2})^{**})
		\cong
		\operatorname{Gr}^2_F H^2(X,\mathbb C).
		\]
		Since \(H^2(X,\mathbb C)\) is pure of weight \(2\), the last term is \(H^{2,0}(X)\).
		Hence the symplectic form on \(X_{\rm reg}\) determines uniquely a class
		\[
		\sigma_X\in H^{2,0}(X).
		\]
	\end{remark}
	
	\begin{definition}[{\cite[Definition 5.2]{BL22}}]\label{def-bbf-fourfold}
		Let $Z$ be a compact complex variety in Fujiki class $\mathscr C$ (Definition \ref{def-fujiki-c}) with
		rational singularities and $\dim Z=2n$. Let
		\[
		\sigma\in H^{2,0}(Z)
		\]
		be the cohomology class associated with a holomorphic two-form on $Z_{\rm reg}$.
		The \emph{associated quadratic form}
		\[
		q_{Z,\sigma}:  H^2(Z,\mathbb C)\longrightarrow \mathbb C
		\]
		is defined by
		\begin{equation}\label{eq-bbf-arbitrary-form}
			q_{Z,\sigma}(\alpha)
			:=
			\frac{n}{2}\int_Z(\sigma\overline{\sigma})^{n-1}\alpha^2
			+
			(1-n)
			\left(\int_Z \sigma^n\overline{\sigma}^{\,n-1}\alpha\right)
			\left(\int_Z \sigma^{n-1}\overline{\sigma}^{\,n}\alpha\right),
			\qquad
			\alpha\in H^2(Z,\mathbb C).
		\end{equation}
		Here $\int_Z$ denotes the cap product with the fundamental class $[Z]$.
		In particular, if $\alpha\in H^{1,1}(Z)$, then the two  factors in the
		second term vanish for Hodge-theoretic type reasons
		(see Remark \ref{rem-singular-hodge-and-type} for the Hodge structure), and hence
		\begin{equation}\label{eq:bbf-arbitrary-11}
			q_{Z,\sigma}(\alpha)
			=
			\frac{n}{2}\int_Z(\sigma\overline{\sigma})^{n-1}\alpha^2.
		\end{equation}
		Clearly,  $q_{Z,\sigma}$ is defined over $\mathbb{R}$, i.e., $q_{Z,\sigma}(\alpha)\in \mathbb{R}$ for any real class $\alpha\in H^2(Z,\mathbb R)$.
		If $(Z,\sigma_Z)$ is further assumed to be a $2n$-dimensional
		hyperfujiki manifold
		(Definition \ref{def-hyperfujiki-hyperka}), then, after normalizing
		$\sigma_Z$ by
		\[
		\int_Z(\sigma_Z\overline{\sigma_Z})^n=1,
		\]
		we fix a positive constant $c_Z>0$ and define the
		\emph{Beauville--Bogomolov--Fujiki (BBF) form} on $Z$ by
		\begin{equation}\label{qz-cz}
			q_Z:=c_Zq_{Z,\sigma_Z}.
		\end{equation}
		Under the above normalization, $q_{Z,\sigma_Z}$ coincides with the
		quadratic form in \cite[Definition 4.2]{ACRT18}.  The choice of $c_Z$ does not affect
		the  sign  of a BBF form in the present paper.
		\footnote{\label{multi-reason}
			For a primitive symplectic variety $X$, the corresponding constant
			$c_X>0$ may be chosen so that $c_Xq_{X,\sigma_X}$ takes integral values
			on $H^2(X,\mathbb Z)$ (\cite[Lemma 5.7 and Remark 5.8]{BL22}).}

	\end{definition}
	
	\begin{definition}[{\cite[Definitions 5.2 and 5.4; paragraph following Definition 5.2]{BL22}}]\label{def-bbf-primitive-fourfold}
		Let $X$ be a primitive symplectic variety
		(Definition \ref{def-symp-varie}) of dimension $2n$, and let
		\[
		\sigma_X\in H^{2,0}(X)
		\]
		be the cohomology class (see Remark \ref{rem-singular-hodge-and-type}) of its symplectic form on $X_{\rm reg}$.
		
		The \emph{Beauville--Bogomolov--Fujiki (BBF) form}  on $X$
		can be described in the following two equivalent ways.
		
		First, one may define it directly (intrinsically) on $X$. After normalizing
		$\sigma_X$ by
		\[
		\int_X(\sigma_X\overline{\sigma_X})^n=1,
		\]
		the \emph{BBF form  on $X$} is (up to scaling for the  reason in  Footnote  \ref{multi-reason}) defined to be  
		\[
		q_X=c_X q_{X,\sigma_X}
		\]
		for some positive constant $c_X>0$.

		Second, one may define it by passing to a resolution. Let
		\[
		\pi:  \widetilde X\longrightarrow X
		\]
		be a resolution of singularities, and let $\widetilde\sigma_X$ be the
		holomorphic extension to $\widetilde X$ of
		$\pi^*(\sigma_X|_{X_{\rm reg}})$. One first defines
		$q_{\widetilde X,\widetilde\sigma_X}$ on $\widetilde X$ by the same formula
		as in \eqref{eq-bbf-arbitrary-form}, with $\sigma$ replaced by
		$\widetilde\sigma_X$, and then restricts it to
		\[
		H^2(X,\mathbb Q)\subset H^2(\widetilde X,\mathbb Q)
		\]
		(see \cite[Lemma 2.1]{BL22} for the inclusion relation, which is induced by
		$\pi^*$). This agrees with the intrinsic definition above: the restriction of $q_{\widetilde X,\widetilde\sigma_X}$ to $\pi^*H^2(X,\mathbb Q)$ agrees with the unscaled intrinsic form $q_{X,\sigma_X}$ (\eqref{eq-bbf-arbitrary-form}); after multiplying both sides by the same positive constant $c_X$ (\eqref{qz-cz}), one obtains the  $q_X$.
		
		Consequently, for every $(1,1)$-class $\alpha\in H^{1,1}(X)$, we have
		\begin{equation}\label{eq-bbf-primitive-arbitrary-11}
			q_X(\alpha)
			=
			c_X\frac{n}{2}
			\int_X\alpha^2(\sigma_X\overline{\sigma_X})^{n-1},
		\end{equation}
		which is also equal to
		\begin{equation}\label{eq-bbf-primitive-arbitrary-resolution}
			q_X(\alpha)
			=
			c_X\frac{n}{2}
			\int_{\widetilde X}(\pi^*\alpha)^2
			(\widetilde\sigma_X\overline{\widetilde\sigma_X})^{n-1}.
		\end{equation}
	\end{definition}

	Note that the BBF square is positive on K\"ahler classes (e.g., \cite[\S 1.10]{Huy99}, \cite[Definition 6.7]{BL22}) and therefore
	on the first Chern class of any ample line bundle. For a big line bundle,
	however, this need not be the case.
	
	\begin{fact}\label{fact-big-notposi}
		There exist a projective hyperk\"ahler manifold $Y$ and a big line bundle
		$L$ on $Y$ such that
		\[
		q_Y\bigl(c_1(L)\bigr)<0.
		\]
	\end{fact}
	
	One way to produce examples as in the above fact is as follows.  Let $Y$
	be a projective hyperk\"ahler manifold, let $L_0$ be a big line bundle on
	$Y$, put $\alpha:=c_1(L_0)$, and suppose that $E$ is an effective divisor
	such that $q_Y([E])<0$, where $[E]:=c_1\bigl(\mathcal O_Y(E)\bigr)$. Let
	$B_{q_Y}$ denote the  symmetric bilinear polarization of $q_Y$.
	For each integer $m\geq 0$, the line bundle
	\[
	L_m:=L_0\otimes\mathcal O_Y(mE)
	\]
	is still big and satisfies $c_1(L_m)=\alpha+m[E]$, whereas
	\[
	q_Y\bigl(c_1(L_m)\bigr)
	=
	q_Y(\alpha)
	+2mB_{q_Y}(\alpha,[E])
	+m^2q_Y([E]).
	\]
	Since $q_Y([E])<0$, the quadratic term dominates the remaining terms for
	all sufficiently large integers $m$. Thus, on any projective
	hyperk\"ahler manifold carrying such an effective divisor, the first
	Chern class of a big line bundle can have negative BBF square.
	
	We now give a precise example. Let
	$Y\subset\mathbb P^3$ be the Fermat quartic (\cite[\S 3, pp. 1942--1943]{SSvL10}), let $H$ be its hyperplane
	class, and let $C\subset Y$ be a line. The surface $Y$ is a smooth K3
	surface, and
	\[
	L:=\mathcal O_Y(H+3C)
	\]
	is big because $H$ is ample and $3C$ is effective. Since $H^2=4$,
	$H\cdot C=1$, and $C^2=-2$, where the last equality follows from
	\cite[Chapter 2, \S 1.3, (1.4)]{Huy16}, we obtain, with the standard
	normalization for which the BBF form of a K3 surface is its intersection
	form,
	\[
	q_Y\bigl(c_1(L)\bigr)
	=
	(H+3C)^2
	=
	4+6-18
	=
	-8.
	\]

	\section{Moishezonness implies BBF positivity}
	\label{sec-arb-dim-moishezon-positive}

	\subsection{Small primitive symplectic models}
	
	Note that a Fujiki manifold with vanishing real first Chern class admits
	a suitable terminal model,  provided it is either Moishezon or of dimension at most $4$.
	
	\begin{lemma}\label{lem-sqf-model}
		Let $Y$ be a Fujiki manifold of dimension $d$ such that
		$c_1(Y)=0\in H^2(Y,\mathbb R)$.
		Assume that one of the following conditions holds:
		\begin{enumerate}[\rm{(}1\rm{)}]
			\item $Y$ is Moishezon;
			\item $\dim Y\leq 4$.
		\end{enumerate}
		Then there exist a normal compact K\"ahler strongly
		$\mathbb Q$-factorial $d$-fold $Y_{\rm{sqf}}$ with terminal singularities
		and a bimeromorphic map
		\[
		\psi:Y\dashrightarrow Y_{\rm{sqf}}
		\]
		such that $\psi$ is an isomorphism in codimension one.
	\end{lemma}
	
	\begin{proof}
		By \cite[Proof of Theorem 3.8, Step 1]{BCDG25}, applied to the
		Fujiki manifold $Y$ with $c_1(Y)=0$ and satisfying either $Y$ is Moishezon
		or $\dim Y\leq 4$, there exists a normal compact K\"ahler  $d$-fold
		$Y_{\min}$ with terminal singularities and a bimeromorphic map
		\[
		\tau:Y\dashrightarrow Y_{\min}
		\]
		such that $\tau$ is an isomorphism in codimension one. 
		We may now repeat verbatim the argument following the proof of Claim \ref{claim-small-BCDG} in the proof of Lemma \ref{lem-sqf-model-KMM} to obtain the desired  $Y_{\rm{sqf}}$ (in Lemma \ref{lem-sqf-model-KMM}, the condition $\left(\mathrm{KMM}_{0,2n}\right)$ is only to obtain the $Y_{\min}$, where $d=2n$).
	\end{proof}
	
	Lemma \ref{lem-sqf-model} provides a small strongly $\mathbb Q$-factorial
	terminal K\"ahler model whenever $Y$   is Moishezon or $\dim Y\leq 4$, provided that $Y$ satisfies $c_1(Y)=0\in H^2(Y,\mathbb R)$.
	We now show that, under the additional hyperfujiki assumption, the 
	symplectic form of $Y$ descends to this model and endows it with the structure
	of a primitive symplectic variety.

	\begin{proposition}[Small strongly $\mathbb Q$-factorial primitive symplectic K\"ahler model]
		\label{prop-small-primitive-model}
		Let $(Y,\sigma_Y)$ be a hyperfujiki manifold of dimension $2n$, where
		$n\geq 2$.  Assume further that one of the following conditions holds:
		\begin{enumerate}[\rm{(}1\rm{)}]
			\item $Y$ is Moishezon;
			\item $n=2$.
		\end{enumerate}
		Then there exist a compact K\"ahler
		strongly $\mathbb Q$-factorial normal variety $X$ with only terminal
		singularities,  a small bimeromorphic map
		\[
		\phi:Y\dashrightarrow X,
		\]
		and a reflexive two-form
		\[
		\sigma_X\in
		H^0\left(X,\left(\wedge^2\Omega_X^1\right)^{**}\right)
		\]
		satisfying 
		\begin{equation}
			\label{eq-kmm-form-new}
			\sigma_Y|_{U_Y}
			=
			\phi_U^*\left(\sigma_X|_{U_X}\right), 
		\end{equation} 
		where  $U_Y\subset Y$ and $U_X\subset X$ are 
		the maximal analytic Zariski open subsets on which $\phi$ induces a
		biholomorphism $\phi_U:U_Y\xrightarrow{\cong}U_X.$
		Moreover,
		\[
		\left(X,\sigma_X|_{X_{\rm{reg}}}\right)
		\]
		is a primitive symplectic variety (Definition \ref{def-symp-varie}).

	\end{proposition}
	
	\begin{proof}

		Since $(Y,\sigma_Y)$ is hyperfujiki,  $c_1(Y)=0\in H^2(Y,\mathbb R).$
		Since either $Y$ is Moishezon or $n=2$,
		Lemma \ref{lem-sqf-model} applies to $Y$. Thus we obtain a compact
		K\"ahler strongly $\mathbb Q$-factorial normal variety $X$ with only
		terminal singularities and a small bimeromorphic map
		\[
		\phi:Y\dashrightarrow X.
		\]
		We may now repeat verbatim the argument following the first paragraph
		of the proof of Proposition \ref{prop-kmm-model} to complete the proof
		of Proposition \ref{prop-small-primitive-model} (in Proposition \ref{prop-kmm-model}, the condition $\left(\mathrm{KMM}_{0,2n}\right)$ is only to obtain the map $\phi:Y\dashrightarrow X$).

	\end{proof}

	\subsection{Moishezonness implies BBF positivity}
	
	We now give a necessary condition for a hyperfujiki manifold to be Moishezon from the viewpoint of BBF positivity. In particular, it implies that if $q_Y(c_1(L_Y))\leq 0$ for any holomorphic line bundle $L_Y$ on $Y$, then $Y$ is not Moishezon.

	\begin{theorem}\label{thm-arb-positive}
		Let $(Y,\sigma_Y)$ be a hyperfujiki $2n$-fold and $q_Y$ the BBF form. Then
		\[
		Y\text{ is Moishezon}
		\Longrightarrow \text{ there exists a big line bundle } L_Y \text{ such that }
		q_Y(c_1(L_Y))>0.
		\]
	\end{theorem}
	
	\begin{proof}
		\setcounter{proofstep}{0}

		Assume that $Y$ is Moishezon. If $n=1$, then $Y$ is a K\"ahler Moishezon
		surface, hence projective. 
		Choose an ample line bundle $L_Y$ on $Y$. Then $L_Y$ is big.  Since the K\"ahler cone of a hyperk\"ahler manifold is contained in the BBF-positive cone (\cite[\S 1.10]{Huy99}  or \cite[Definition 6.7]{BL22}), we have 
		\[
		q_Y(c_1(L_Y))>0.
		\]
		Thus the theorem holds when $n=1$.  Thus, from now on, we assume that $n\geq 2$.

		The proof is divided into 3 steps. First, we fix the small primitive
		symplectic model and prove the  boundary power vanishing (which is needed in Step 2). Second, we
		compare the BBF forms for the transform of a line bundle from $X$ to $Y$.
		Third, we apply the comparison to an ample line bundle on the projective model
		$X$.

		\begin{step}\label{step 1-thm-arb-positive}
			Vanishing of the $(n-1)$-power of the pulled-back symplectic form along certain boundary divisors
		\end{step}

		Apply Proposition \ref{prop-small-primitive-model} and we obtain a compact
		K\"ahler strongly $\mathbb Q$-factorial primitive symplectic variety $X$ with
		only terminal singularities and a small bimeromorphic map
		\[
		\phi:Y\dashrightarrow X.
		\]
		Furthermore, the symplectic form (Definition \ref{def-symp-varie}) on $X_{\rm reg}$ is 
		$\sigma_X|_{X_{\rm reg}}$, where 
		\[
		\sigma_X\in
		H^0\left(X,\left(\wedge^2\Omega_X^1\right)^{**}\right).
		\]
		
		Choose maximal analytic Zariski open subsets $U_Y\subset Y$ and $U_X\subset X$ such
		that $\phi$ induces an isomorphism
		\[
		\phi_U:U_Y\xrightarrow{\cong}U_X
		\]
		and
		\begin{equation}\label{codim-geq-2-thm}
			\operatorname{codim}_Y(Y\setminus U_Y)\geq 2,
			\qquad
			\operatorname{codim}_X(X\setminus U_X)\geq 2.
		\end{equation}
		Since $U_Y$ is smooth and $U_X\cong U_Y$, we have
		$U_X\subset X_{\rm reg}$.

		Let $\Gamma$ be the graph of $\phi$, endowed with the reduced structure, and
		let
		\[
		p:\Gamma\to Y,
		\qquad
		q:\Gamma\to X
		\]
		be the natural projections. Choose a resolution
		\[
		\rho:W\to\Gamma
		\]
		which is biholomorphic over $\Gamma_{\rm reg}$. Set
		\[
		\mu:=p\circ\rho:W\to Y,
		\qquad
		\nu:=q\circ\rho:W\to X.
		\]
		Thus $W$ is a common resolution of 
		$Y$ and $X$. Define
		\begin{equation}\label{def-Gamma-U-thm}
			\Gamma_U:=p^{-1}(U_Y)=q^{-1}(U_X),
			\qquad
			W^\circ:=\nu^{-1}(X_{\rm reg}),
			\qquad
			W_U:=\rho^{-1}(\Gamma_U)
			=\mu^{-1}(U_Y)
			=\nu^{-1}(U_X).
		\end{equation}
		Then
		\[
		\mu|_{W_U}:W_U\xrightarrow{\cong}U_Y,
		\qquad
		\nu|_{W_U}:W_U\xrightarrow{\cong}U_X.
		\]
		Set
		\begin{equation}\label{def-tau}
			\tau:=\mu^*\sigma_Y\in H^0(W,\Omega_W^2).
		\end{equation}
		By the construction of $\sigma_X$ in Proposition
		\ref{prop-small-primitive-model}, the form $\tau$ is the extension to $W$ of
		\[
		(\nu|_{W^\circ})^*(\sigma_X|_{X_{\rm reg}}).
		\]

		Philosophically motivated by the    coisotropicity-type (\cite[Theorem 1.2]{Wie03}) and isotropicity-type (\cite[Lemma 2.1]{WW03} or  \cite[Chapter 4]{Wie00} or \cite[\S 4]{Kl01}) phenomena, we establish the following claim, which is the key ingredient in comparing the BBF forms of $Y$ and $X$ (see Step 2).
		However, our setting and conclusion
		are substantially different from those of the cited results. In particular, Claim
		\ref{claim-boundary-power-vanishing} is not a direct consequence of these results, and their arguments
		do not seem to adapt directly to the present situation.

		\begin{claim}\label{claim-boundary-power-vanishing}
			Let $E\subset W\setminus W_U$ be a reduced and irreducible  divisor, let
			$\iota_E:E\hookrightarrow W$ be the natural inclusion, and let
			\[
			r:\widetilde E\to E
			\]
			be a resolution. Then
			\[
			\left((\iota_E\circ r)^*\tau\right)^{n-1}=0.
			\]
		\end{claim}
		
		\begin{proof}[Proof of Claim \ref{claim-boundary-power-vanishing}]
			This proof is divided into 3 steps.
			
			\medskip
			
			\noindent\textbf{Step (1).} Reduction to the codimension-$2$ case and construction of $F$
			
			\medskip
			
			We first treat the case in which the reduced image of $E$ under $\mu$ has
			codimension at least $3$ in $Y$, where the required vanishing follows
			directly from the dimension reason. In the remaining codimension-$2$ case, we pass to
			the normalization of the graph of $\phi$ and construct a prime divisor $F$
			dominating this image.
			
			Set $S:=\mu(E)_{\rm red}.$
			Since $\mu|_E$ is proper and $E$ is irreducible, $S$ is an irreducible
			analytic subset of $Y$. Moreover, since $E\subset W\setminus W_U$, we have
			$S\subset Y\setminus U_Y$. By \eqref{codim-geq-2-thm},
			\[
			\dim S\leq 2n-2.
			\]
			
			If $\dim S\leq 2n-3$, then the image of
			\[
			\mu\circ\iota_E\circ r:\widetilde E\to Y
			\]
			has dimension at most $2n-3$. Consequently, the pull-back of the holomorphic
			$(2n-2)$-form $\sigma_Y^{n-1}$ vanishes on a dense open subset of
			$\widetilde E$, and therefore vanishes everywhere. Thus
			\[
			\left((\iota_E\circ r)^*\tau\right)^{n-1}
			=
			(\mu\circ\iota_E\circ r)^*(\sigma_Y^{n-1})
			=0.
			\]
			It remains to treat the case
			\[
			\dim S=2n-2.
			\]
			
			Let
			\[
			n_\Gamma:\Gamma^\nu\to\Gamma
			\]
			be the normalization    of the graph $\Gamma$ of $\phi$. Set
			\[
			a:=p\circ n_\Gamma:\Gamma^\nu\to Y,
			\qquad
			b:=q\circ n_\Gamma:\Gamma^\nu\to X.
			\]
			Since $W$ is normal, the morphism $\rho:W\to\Gamma$ factors through the
			normalization, based on the universality of the normalization. Thus there exists a morphism $\gamma:W\to\Gamma^\nu$ such that
			\[
			\rho=n_\Gamma\circ\gamma,
			\qquad
			\mu=a\circ\gamma,
			\qquad
			\nu=b\circ\gamma.
			\]
			The morphisms introduced above are summarized in the following diagram:
			\begin{equation}\label{diag-res}
				\begin{tikzcd}[column sep=large, row sep=large]
					& W
					\arrow[d, "\gamma"]
					\arrow[dddl, bend right=18, "\mu"']
					\arrow[dddr, bend left=18, "\nu"]
					& \\
					& \Gamma^\nu
					\arrow[d, "n_\Gamma"]
					\arrow[ddl, bend right=12, "a"']
					\arrow[ddr, bend left=12, "b"]
					& \\
					& \Gamma
					\arrow[dl, "p"']
					\arrow[dr, "q"]
					& \\
					Y
					\arrow[rr, dashed, "\phi"']
					& & X .
				\end{tikzcd}
			\end{equation}
			Identifying $\Gamma$ with its image in $Y\times X$, the finite morphism
			$n_\Gamma$ is given by
			\[
			n_\Gamma=(a,b):\Gamma^\nu\to\Gamma\subset Y\times X.
			\]
			Moreover,
			\[
			n_\Gamma^{-1}(\Gamma_U)
			=a^{-1}(U_Y)
			=b^{-1}(U_X).
			\]
			
			We first show that $S$ is contained in the indeterminacy locus
			$\mathcal S_\phi$ of $\phi$. Suppose otherwise, and choose a general point
			\[
			s\in S_{\rm reg}\setminus\mathcal S_\phi.
			\]
			Then $\phi$ is represented by a holomorphic map on a neighbourhood
			$V_s\subset Y$ of $s$, and the graph projection gives an isomorphism
			\[
			p^{-1}(V_s)\xrightarrow{\cong}V_s.
			\]
			In particular, $p^{-1}(V_s)$ is smooth. Since $\rho$ is biholomorphic over
			$\Gamma_{\rm reg}$, the induced map
			\[
			\mu|_{\mu^{-1}(V_s)}:\mu^{-1}(V_s)\xrightarrow{\cong}V_s
			\]
			is biholomorphic. Since $s\in\mu(E)$, the set
			$E\cap\mu^{-1}(V_s)$ is a nonempty divisor in $\mu^{-1}(V_s)$. Its image in
			$V_s$ is therefore a divisor, but this image is contained in $S\cap V_s$,
			which has dimension $2n-2$. This is a contradiction. Thus
			\[
			S\subset\mathcal S_\phi.
			\]
			
			Note that for each $s\in S$, we have
			\[
			\dim p^{-1}(s)>0,
			\]
			and thus it follows from the finiteness of $n_\Gamma$ that 
			\[
			\dim a^{-1}(s)>0
			\]
			for each $s\in S$.
			Indeed, if $p^{-1}(s)$ were zero-dimensional, upper semicontinuity of fiber
			dimension \cite[Theorem 1.19]{PR94} would give a neighbourhood $V_s$ of $s$
			such that $p^{-1}(V_s)\to V_s$ has zero-dimensional fibers. Since $p$ is
			proper, this restriction would be finite. Since it is also a proper modification,  it would be an isomorphism by Zariski's main theorem
			\cite[Theorem 1.11]{Ue75}, and $\phi$ would be holomorphic near $s$, a
			contradiction. 
			
			Take an irreducible component
			\[
			F\subset a^{-1}(S)
			\]
			which dominates $S$ and whose general fiber over $S$ is positive-dimensional. Then
			\[
			\dim F\geq\dim S+1=2n-1.
			\]
			On the other hand, $a^{-1}(S)$ is a proper analytic subset of the
			$2n$-dimensional irreducible  space $\Gamma^\nu$. 
			Consequently,
			\[
			\dim F=2n-1.
			\]
			Since $\Gamma^\nu$ is normal, $F$ is not contained in the singular locus
			$(\Gamma^\nu)_{\rm sing}$ of $\Gamma^\nu$.
			
			Because $F\subset a^{-1}(Y\setminus U_Y)$ and
			$a^{-1}(U_Y)=b^{-1}(U_X)$, we have
			\[
			b(F)\subset X\setminus U_X.
			\]
			Put
			\[
			T:=b(F)_{\rm red}.
			\]
			Then $T$ is an irreducible analytic subset of $X$ and
			\[
			\dim T\leq2n-2.
			\]

			\medskip
			
			\noindent\textbf{Step (2).} The  pull-back formula on $F^\circ$
			
			\medskip
			
			We now compare the pull-back
			of $\sigma_Y$ with the reflexive pull-back of $\sigma_X$ on the regular locus of the normalized graph. After restricting
			to suitable dense open subsets, we obtain two descriptions of
			$\eta|_{F^\circ}$: one as a pull-back from $S^\circ$ via $a^\circ$, and the
			other as a pull-back from $T^\circ$ via $b^\circ$.

			Set  $G:=(\Gamma^\nu)_{\rm reg}$
			and define
			\begin{equation}\label{def-eta}
				\eta:=(a|_G)^*\sigma_Y
				\in H^0(G,\Omega_G^2).
			\end{equation}
			Recall that we have the reflexive two-form
			\[
			\sigma_X\in
			H^0\left(X,\left(\wedge^2\Omega_X^1\right)^{**}\right).
			\]
			Since $X$  has only   terminal singularities,   it has only rational singularities (this  also follows from
			\cite[Theorem 3.4(1)]{BL22}, since $X$ is a symplectic variety). The source $G$  also has only
			rational singularities. Thus the analytic pull-back theorem for reflexive
			differential forms (\cite[Theorems 1.10, 14.1 and Fact 14.3 (14.1.3)]{KS21}) 
			defines \footnote{When $T \nsubseteq X_{\rm {sing }}$, we can also use the usual standard pullback over the corresponding regular part.}
			\[
			\eta_b
			:=
			d_{\rm refl}(b|_G)(\sigma_X)
			\in H^0(G,\Omega_G^2).
			\]
			
			The open subset
			\[
			G_U:=n_\Gamma^{-1}(\Gamma_U)
			\]
			is a dense open subset of $G$, because $\Gamma_U$ is smooth and the
			normalization morphism is an isomorphism over $\Gamma_U$. On $G_U$, the
			compatibility of reflexive pull-back with ordinary pull-back gives
			\[
			\begin{aligned}
				\eta_b|_{G_U}
				&=
				(b|_{G_U})^*(\sigma_X|_{U_X})                       \\
				&=
				(a|_{G_U})^*\left(\phi_U^*(\sigma_X|_{U_X})\right) \\
				&=
				(a|_{G_U})^*(\sigma_Y|_{U_Y})                       \\
				&=
				\eta|_{G_U}.
			\end{aligned}
			\]
			Since $G_U$ is dense in $G$, the identity theorem gives
			\begin{equation}\label{eq-eta-reflexive-pullback}
				\eta
				=
				d_{\rm refl}(b|_G)(\sigma_X)
				\text{ on }G.
			\end{equation}

			Apply the generic submersion theorem
			\cite[Proposition 1.21]{PR94} to $a|_F: F\to S$, $b|_F: F\to T$, and
			$n_\Gamma|_F: F\to n_\Gamma(F)$, where both $F$ and the irreducible analytic subset $n_\Gamma(F)$ of $\Gamma$ are  equipped with the reduced structures. Since $n_\Gamma|_F$ is finite onto its image, it is generally locally biholomorphic and its
			rank at general point is equal to $\dim F$. Thus, after shrinking to a dense open subset, we
			may choose
			\[
			F^\circ\subset F_{\rm reg}\cap G
			\]
			such that
			\[
			S^\circ:=a(F^\circ)\subset S_{\rm reg},
			\qquad
			T^\circ:=b(F^\circ)\subset T_{\rm reg}
			\]
			are dense  open subsets, the induced morphisms
			\[
			a^\circ:=a|_{F^\circ}:F^\circ\to S^\circ,
			\qquad
			b^\circ:=b|_{F^\circ}:F^\circ\to T^\circ
			\]
			are submersions, and
			\[
			d(n_\Gamma|_{F^\circ})
			\]
			is injective.
			
			Let
			\[
			i_S:S^\circ\hookrightarrow Y,
			\qquad
			i_T:T^\circ\hookrightarrow X
			\]
			be the natural locally closed immersions. Since $T^\circ$ is smooth, the
			analytic pull-back theorem (\cite[Theorem 14.1]{KS21}) defines
			\[
			\omega_T
			:=
			d_{\rm refl}i_T(\sigma_X)
			\in H^0(T^\circ,\Omega_{T^\circ}^2).
			\]
			Note that this construction remains valid even when $T\subset X_{\rm sing}$, since the
			reflexive pull-back is  also defined for morphisms whose image is entirely contained
			in the singular locus of the target
			(\cite[Theorem 1.10 and Note 1.6.1]{KS21}).
			
			Let
			\[
			\iota_F:F^\circ\hookrightarrow G
			\]
			be the natural inclusion. By construction, we have
			\begin{equation}\label{factori-1}
				i_S\circ a^\circ=(a|_G)\circ\iota_F   
			\end{equation}
			and
			\begin{equation}\label{factori-2}
				i_T\circ b^\circ=(b|_G)\circ\iota_F.
			\end{equation}

			First, by the definition \eqref{def-eta} of $\eta$ and \eqref{factori-1}, we have
			\[
			\begin{aligned}
				\eta|_{F^\circ}
				&:=
				\iota_F^*
				\left(
				(a|_G)^*\sigma_Y
				\right)                                                      \\
				&=
				(a^\circ)^*
				\left(
				i_S^*\sigma_Y
				\right).                      
			\end{aligned}
			\]
			On the other hand, by \eqref{eq-eta-reflexive-pullback}, the functoriality  (\cite[Theorem 14.1]{KS21}) of the
			reflexive pull-back functor $ d_{\rm refl}$, and \eqref{factori-2}, we obtain
			\[
			\begin{aligned}
				\eta|_{F^\circ}
				&:=
				\iota_F^*\eta                                                \\
				&=
				d_{\rm refl}\iota_F
				\left(
				d_{\rm refl}(b|_G)(\sigma_X)
				\right)                                                      \\
				&=
				d_{\rm refl}
				\left(
				(b|_G)\circ\iota_F
				\right)(\sigma_X)                                            \\
				&=
				d_{\rm refl}
				\left(
				i_T\circ b^\circ
				\right)(\sigma_X)                                            \\
				&=
				d_{\rm refl}b^\circ
				\left(
				d_{\rm refl}i_T(\sigma_X)
				\right)                                                      \\
				&=
				d_{\rm refl}b^\circ(\omega_T)                                \\
				&=
				(b^\circ)^*\omega_T.
			\end{aligned}
			\]
			Here the second equality and the final equality follow from compatibility (\cite[Theorem 14.1 and Fact 14.3 (14.1.3)]{KS21}) with
			ordinary pull-back, since the respective targets $G$ and $T^\circ$ are smooth.  \footnote{Note that if $T\subset X_{\rm sing}$, then $i_T^{-1}(X_{\rm reg})=\varnothing.$
				Accordingly, we cannot  use the regular-locus compatibility diagram of
				\cite[Theorem 14.1]{KS21} for the morphism $i_T$. }
			Consequently,
			\begin{equation}\label{eq-graph-form-arb}
				\eta|_{F^\circ}
				=
				(a^\circ)^*(\sigma_Y|_{S^\circ})
				=
				(b^\circ)^*\omega_T.
			\end{equation}

			\medskip
			
			\noindent\textbf{Step (3).} The kernel estimate and  the vanishing
			
			\medskip

			We now use these two pull-back descriptions obtained in Step (2) to show that
			the relative tangent bundles of $a^\circ$ and $b^\circ$ lie in the kernel of
			$\eta|_{F^\circ}$.
			The resulting kernel estimate forces
			$\left(\eta|_{F^\circ}\right)^{n-1}=0$. This vanishing descends to
			$S_{\rm reg}$ and then pulls back to the desired vanishing on
			$\widetilde E$.

			Set
			\[
			K_a:=T_{F^\circ/S^\circ}:=\ker(d a^\circ),
			\qquad
			K_b:=T_{F^\circ/T^\circ}:=\ker(d b^\circ)
			\]
			where 
			\[
			d a^\circ:
			T_{F^\circ}
			\longrightarrow
			(a^\circ)^*T_{S^\circ},
			\qquad
			d b^\circ:
			T_{F^\circ}
			\longrightarrow
			(b^\circ)^*T_{T^\circ}.
			\]
			Since
			\[
			\dim F^\circ=2n-1,
			\qquad
			\dim S^\circ=2n-2,
			\qquad
			\dim T^\circ\leq2n-2,
			\]
			we have
			\[
			\operatorname{rank}K_a=1,
			\qquad
			\operatorname{rank}K_b\geq1.
			\]
			Moreover, since $n_\Gamma=(a,b)$ and
			$d(n_\Gamma|_{F^\circ})_x$ is injective for each $x\in F^\circ$,  we have
			$$
			\begin{aligned}
				\left(K_a\right)_x \cap\left(K_b\right)_x & =\operatorname{ker} d a_x^{\circ} \cap \operatorname{ker} d b_x^{\circ} \\
				& =\operatorname{ker} d\left(\left. n_{\Gamma}\right|_{F^{\circ}}\right)_x \\
				& =\{0\}
			\end{aligned}
			$$
			for each $x\in F^\circ$.
			By \eqref{eq-graph-form-arb}, both $K_a$ and $K_b$ are contained in the kernel
			of $\eta|_{F^\circ}$, where the two-form $\eta|_{F^\circ}$ is viewed by contraction as a
			morphism $T_{F^\circ}\to\Omega_{F^\circ}^1.$
			Consequently, for each $x\in F^\circ$, we have
			\[
			\dim\ker\left(\eta_x|_{T_xF^\circ}\right)\geq2.
			\]
			Since the rank of a skew-symmetric form is even, and $\dim T_xF^\circ=2n-1$ is
			odd, the dimension of its kernel is odd. Therefore, we have
			\[
			\dim\ker\left(\eta_x|_{T_xF^\circ}\right)\geq3,
			\]
			and thus
			\[
			\operatorname{rank}\left(\eta_x|_{T_xF^\circ}\right)
			\leq 2n-4.
			\]
			It then follows that
			\[
			\left(\eta|_{F^\circ}\right)^{n-1}=0.
			\]

			Using \eqref{eq-graph-form-arb}, we obtain
			\[
			(a^\circ)^*\left((\sigma_Y|_{S^\circ})^{n-1}\right)
			=
			\left(\eta|_{F^\circ}\right)^{n-1}
			=0.
			\]
			Since $a^\circ:F^\circ\to S^\circ$ is a surjective submersion, pull-back of holomorphic
			forms is pointwise injective (e.g., \cite[\S 18, Problem 18.8]{Tu11}). Thus
			\[
			(\sigma_Y|_{S^\circ})^{n-1}=0.
			\]
			Since $S^\circ$ is dense in $S_{\rm reg}$, the identity theorem gives
			\[
			(\sigma_Y|_{S_{\rm reg}})^{n-1}=0.
			\]
			
			Finally, since $E$ dominates $S$, the inverse image of $S_{\rm reg}$ under
			\[
			\mu\circ\iota_E\circ r:\widetilde E\to S
			\]
			is dense in $\widetilde E$. On this dense open subset, we have
			\[
			\left((\iota_E\circ r)^*\tau\right)^{n-1}
			=
			(\mu\circ\iota_E\circ r)^*(\sigma_Y^{n-1})
			=0.
			\]
			The left-hand side is a holomorphic $(2n-2)$-form on $\widetilde E$.
			Therefore it vanishes everywhere. This completes the proof of  Claim \ref{claim-boundary-power-vanishing}.
		\end{proof}

		\begin{step}
			Transformation of line bundles from $X$ to $Y$ and comparison of the BBF forms
		\end{step}

		Let $A\in\operatorname{Pic}(X)$ be a holomorphic line bundle. 
		Define 
		\[
		\mathcal F:=\mu_*\nu^*A,
		\qquad
		L_Y=\mathcal F^{**},
		\]
		where $\mu$ and $\nu$ are the morphisms appearing in
		Diagram \eqref{diag-res}.  Since $Y$ is smooth, the rank-one reflexive sheaf $L_Y$ is a line bundle (e.g., \cite[Chapter II, \S 1.1, Lemma 1.1.15]{OSS11}).
		Consider the  morphism $$\ell:\mu^*\mathcal F\longrightarrow\mu^*L_Y$$  induced by 
		the canonical morphism $\mathcal F\to\mathcal F^{**}$, and   the
		adjunction morphism $$\varepsilon:\mu^*\mathcal F\longrightarrow\nu^*A.$$
		Both morphisms are isomorphisms over $W_U$ (see \eqref{def-Gamma-U-thm}). Hence their kernels and
		cokernels are torsion sheaves supported on $W\setminus W_U$.
		
		Let $\mathcal M_W$ denote the sheaf of germs of meromorphic functions
		on $W$. Since $\mathcal M_W$ is stalkwise the localization of
		$\mathcal O_W$ by its non-zero divisors, tensoring with $\mathcal M_W$
		annihilates torsion and gives isomorphisms of $\mathcal M_W$-module sheaves
		\[
		\ell_{\mathcal M}:
		\mu^*\mathcal F\otimes_{\mathcal O_W}\mathcal M_W
		\xrightarrow{\cong}
		\mu^*L_Y\otimes_{\mathcal O_W}\mathcal M_W
		\]
		and
		\[
		\varepsilon_{\mathcal M}:
		\mu^*\mathcal F\otimes_{\mathcal O_W}\mathcal M_W
		\xrightarrow{\cong}
		\nu^*A\otimes_{\mathcal O_W}\mathcal M_W.
		\]
		Therefore
		\[
		\Phi:=
		\ell_{\mathcal M}\circ\varepsilon_{\mathcal M}^{-1}
		:
		\nu^*A\otimes_{\mathcal O_W}\mathcal M_W
		\xrightarrow{\cong}
		\mu^*L_Y\otimes_{\mathcal O_W}\mathcal M_W
		\]
		is an isomorphism of $\mathcal M_W$-module sheaves.
		
		Set
		\[
		N:=
		\mu^*L_Y\otimes_{\mathcal O_W}(\nu^*A)^*.
		\]
		Since $\nu^*A$  and   $\mu^*L_Y$  are invertible, we have the canonical isomorphism
		\[
		\mathcal{H}om_{\mathcal M_W}
		\left(
		\nu^*A\otimes_{\mathcal O_W}\mathcal M_W,
		\mu^*L_Y\otimes_{\mathcal O_W}\mathcal M_W
		\right)
		\cong
		N\otimes_{\mathcal O_W}\mathcal M_W,
		\]
		and under this  isomorphism, $\Phi$ corresponds to a meromorphic section
		\[
		s\in
		H^0\left(
		W,
		N\otimes_{\mathcal O_W}\mathcal M_W
		\right).
		\]
		Since $\Phi$ restricts to a holomorphic isomorphism over $W_U$, the section
		$s$ is holomorphic and nowhere vanishing on $W_U$.
		Consequently, we have
		\[
		\operatorname{div}(s)=\sum_i a_iE_i,
		\qquad a_i\in\mathbb Z,
		\]
		where the $E_i$ are certain prime divisors contained in
		$W\setminus W_U$.
		
		Put
		\[
		\alpha_X:=c_1(A)\in H^2(X,\mathbb Z)\cap H^{1,1}(X),
		\qquad
		\alpha_Y:=c_1(L_Y)\in H^2(Y,\mathbb Z)\cap H^{1,1}(Y).
		\]
		Then taking first Chern classes  of $N$ gives
		\begin{equation}\label{eq-boundary-difference-X-to-Y}
			\mu^*\alpha_Y-\nu^*\alpha_X
			=
			\sum_i a_i[E_i]
			\quad\text{in}\quad
			H^2(W,\mathbb Z)\cap H^{1,1}(W).
		\end{equation}
		
		By \eqref{eq-boundary-difference-X-to-Y}, we obtain
		\[
		\begin{aligned}
			&
			\int_W
			\left((\mu^*\alpha_Y)^2-(\nu^*\alpha_X)^2\right)
			\wedge\tau^{n-1}\wedge\overline{\tau}^{\,n-1}
			\\
			&=
			\sum_i a_i
			\int_W
			[E_i]\wedge(\mu^*\alpha_Y+\nu^*\alpha_X)
			\wedge\tau^{n-1}\wedge\overline{\tau}^{\,n-1}.
		\end{aligned}
		\]
		Let $B$ be a smooth closed representative of
		$\mu^*\alpha_Y+\nu^*\alpha_X$. For each $i$, let
		\[
		r_i:\widetilde E_i\to E_i
		\]
		be a resolution and let $\iota_{E_i}:E_i\hookrightarrow W$ be the natural
		inclusion. By Claim \ref{claim-boundary-power-vanishing},  we obtain
		\[
		\begin{aligned}
			&
			\int_W
			[E_i]\wedge B\wedge\tau^{n-1}\wedge\overline{\tau}^{\,n-1}
			\\
			&=
			\int_{\widetilde E_i}
			(\iota_{E_i}\circ r_i)^*B
			\wedge
			\left((\iota_{E_i}\circ r_i)^*\tau\right)^{n-1}
			\wedge
			\overline{\left((\iota_{E_i}\circ r_i)^*\tau\right)}^{\,n-1}
			=
			0.
		\end{aligned}
		\]
		Therefore
		\begin{equation}\label{eq-integral-general-step}
			\int_W(\mu^*\alpha_Y)^2\wedge\tau^{n-1}\wedge\overline{\tau}^{\,n-1}
			=
			\int_W(\nu^*\alpha_X)^2\wedge\tau^{n-1}\wedge\overline{\tau}^{\,n-1}.
		\end{equation}
		
		Since $\tau=\mu^*\sigma_Y$, the projection formula for the proper modification
		$\mu:W\to Y$ gives
		\begin{equation}\label{eq-integral-Y-general-step}
			\int_W(\mu^*\alpha_Y)^2\wedge\tau^{n-1}\wedge\overline{\tau}^{\,n-1}
			=
			\int_Y
			\alpha_Y^2\wedge
			\sigma_Y^{n-1}\wedge\overline{\sigma_Y}^{\,n-1}.
		\end{equation}
		On the other hand, $\tau$ is the extension (see \eqref{def-tau}) to $W$ of the pull-back of
		$\sigma_X|_{X_{\rm reg}}$ under $\nu$. Since $\nu:W\to X$ is a resolution and one can calculate the BBF via the resolution (Definition \ref{def-bbf-primitive-fourfold}), \eqref{eq-bbf-primitive-arbitrary-resolution} and
		\eqref{eq-bbf-primitive-arbitrary-11}
		give
		\begin{equation}\label{eq-integral-X-general-step}
			\int_W(\nu^*\alpha_X)^2\wedge\tau^{n-1}\wedge\overline{\tau}^{\,n-1}
			=
			\int_X
			\alpha_X^2\,
			(\sigma_X\overline{\sigma_X})^{n-1}.
		\end{equation}
		Combining \eqref{eq-integral-general-step}, \eqref{eq-integral-Y-general-step},
		and \eqref{eq-integral-X-general-step}, we obtain
		\begin{equation}\label{eq-bbf-arb-step}
			\int_Y
			\alpha_Y^2\wedge
			\sigma_Y^{n-1}\wedge\overline{\sigma_Y}^{\,n-1}
			=
			\int_X
			\alpha_X^2\,
			(\sigma_X\overline{\sigma_X})^{n-1}.
		\end{equation}
		By the Hodge type reason and \eqref{eq-bbf-primitive-arbitrary-11}, 
		we obtain 
		\begin{equation}\label{eq-step2-compare-X-to-Y}
			q_Y(c_1(L_Y))=\lambda q_X(c_1(A)),
			\qquad
			\lambda:=\frac{c_Y}{c_X}>0.
		\end{equation}
		The constant $\lambda$ is independent of $A$.
		
		\begin{step}
			Moishezonness gives a BBF-positive integral  class
		\end{step}

		Since $Y$ is Moishezon and $X$ is bimeromorphic to $Y$, the variety $X$ is also
		Moishezon. Moreover, $X$ is K\"ahler and has only terminal singularities. Hence
		$X$ has rational, in particular $1$-rational, singularities. By Namikawa's
		projectivity criterion for compact K\"ahler Moishezon varieties with
		$1$-rational singularities \cite[Theorem 1.6]{Nam02}, the variety $X$ is
		projective.
		
		Choose an ample line bundle $A$ on the primitive symplectic variety $X$. Since the K\"ahler cone of a primitive
		symplectic variety is contained in the BBF-positive cone
		\cite[Definition 6.7]{BL22}, we have
		\[
		q_X(c_1(A))>0.
		\]
		Define
		\[
		L_Y:=(\mu_*\nu^*A)^{**}.
		\]
		Then $L_Y$ is a holomorphic line bundle on $Y$. 
		Note that the line bundle \(L_Y\) is big. Indeed, set \(M:=\nu^*A\). Since \(Y\) is smooth and \(\mu:  W\to Y\) is a
		modification, there exists an open subset \(U\subset Y\)    such that
		\(\operatorname{codim}_Y(Y\setminus U)\geq 2\) and
		\(\mu^{-1}(U)\to U\) is an isomorphism. By the definition of 
		\(L_Y\), one has
		\[
		M|_{\mu^{-1}(U)}\cong L_Y|_U.
		\]
		Consequently, for any \(m\geq 1\), restriction to \(\mu^{-1}(U)\),
		followed by extension across \(Y\setminus U\), gives an injection
		\[
		H^0\bigl(W,M^{\otimes m}\bigr)
		\hookrightarrow
		H^0\bigl(Y,L_Y^{\otimes m}\bigr).
		\]
		Since \(X\) is normal, \(\nu_*\mathcal O_W=\mathcal O_X\); thus the
		projection formula yields
		\[
		h^0\bigl(Y,L_Y^{\otimes m}\bigr)
		\geq
		h^0\bigl(W,\nu^*(A^{\otimes m})\bigr)
		=
		h^0\bigl(X,A^{\otimes m}\bigr).
		\]
		It then follows from the ampleness of $A$ that \(L_Y\) is big.

		Set
		\[
		v:=c_1(L_Y)\in H^2(Y,\mathbb Z)\cap H^{1,1}(Y).
		\]
		Applying \eqref{eq-step2-compare-X-to-Y} to this ample line bundle $A$, we get
		\[
		q_Y(v)=\lambda q_X(c_1(A)).
		\]
		Since $\lambda>0$ and $q_X(c_1(A))>0$, it follows that
		\[
		q_Y(v)>0.
		\]
		This completes the proof of Theorem \ref{thm-arb-positive}.

	\end{proof}

	\section{BBF positivity implies Moishezonness}\label{sec-posi-mois}
	
	\subsection{The criterion of Moishezonness  via small primitive symplectic models}

	\begin{theorem}\label{thm-positive-moishezon}
		Let $(Y,\sigma_Y)$ be a hyperfujiki $2n$-fold, and let $q_Y$ be its
		BBF form. Assume that there exist a compact K\"ahler strongly
		$\mathbb Q$-factorial  (Definition \ref{def-strong-Q-fac}) normal variety $X$ with only terminal singularities
		and a  small bimeromorphic map $\phi:Y\dashrightarrow X.$ 
		Let $ U_Y\subset Y,
		U_X\subset X$
		be the maximal analytic Zariski open subsets such that $\phi$ restricts
		to a biholomorphic map $\phi_U:=\phi|_{U_Y}:U_Y\xrightarrow{\cong}U_X.$
		Assume further that there exists a reflexive
		two-form
		\[
		\sigma_X
		\in
		H^0\left(X,\left(\wedge^2\Omega_X^1\right)^{**}\right)
		\]
		such that
		\begin{equation}\label{eq-form-comp}
			\phi_U^*\left(\sigma_X|_{U_X}\right)
			=
			\sigma_Y|_{U_Y},
		\end{equation}
		and such that
		\[
		\left(X,\sigma_X|_{X_{\rm reg}}\right)
		\]
		is a primitive symplectic variety in the sense of
		Definition \ref{def-symp-varie}.

		Assume moreover that
		there exists
		\[
		v\in H^2(Y,\mathbb Z)\cap H^{1,1}(Y)
		\]
		such that $q_Y(v)>0,$
		then $Y$ is Moishezon.
	\end{theorem}

	\begin{proof}
		Note that the first Betti number of the hyperfujiki manifold $Y$ is even.  Consequently, if $n=1$,
		then $Y$
		is automatically K\"ahler and thus it is a hyperk\"ahler manifold. \footnote{This can also be obtained by the singularity theory: since the symplectic variety $X$ has terminal  singularities, $X_{\rm{sing}}$ has
			codimension at least $4$ (\cite[Theorem 3.4(3)]{BL22} in the analytic setting).   Consequently, 
			if $n=1$, then the terminal surface $X$ is smooth, and a small
			bimeromorphic map between smooth compact surfaces is biholomorphic.}    As a result, the projectivity criterion (\cite[Theorem 2]{Huy01}, \cite[Theorem 3.11]{Huy99} or \cite[Theorem 1.2]{BL22}) for primitive
		symplectic varieties  and the BBF-positivity condition of $v$
		imply that  $Y$ is projective.
		Thus we  assume that $n\geq2$ from now on.

		\setcounter{proofstep}{0}
		
		The proof is divided into three steps. First, we recall the common
		resolution and boundary power vanishing. Second, we use strong
		$\mathbb Q$-factoriality to transform a line bundle from $Y$ to $X$ and
		compare the two BBF forms. Third, we apply the projectivity criterion for
		primitive symplectic varieties.
		
		\begin{step}
			The common resolution and boundary power vanishing
		\end{step}
		
		The argument in Step 1 of the proof of
		Theorem \ref{thm-arb-positive} applies after omitting the initial
		application of Proposition \ref{prop-small-primitive-model}. Indeed, that
		application is used only to produce the variety $X$, the small
		bimeromorphic map $\phi$, and the  reflexive symplectic form,
		all of which are assumed here. No other part of that step uses the
		Moishezonness of $Y$.
		
		Since $\phi$ is small,
		\[
		\operatorname{codim}_Y(Y\setminus U_Y)\geq2,
		\qquad
		\operatorname{codim}_X(X\setminus U_X)\geq2.
		\]
		Moreover, $U_X\subset X_{\rm reg}$ because $U_X\cong U_Y$ and $Y$ is
		smooth. Choose the common resolution used in that step,
		\[
		\mu:W\to Y,
		\qquad
		\nu:W\to X,
		\]
		and set
		\[
		W^\circ:=\nu^{-1}(X_{\rm reg}),
		\qquad
		W_U:=\mu^{-1}(U_Y)=\nu^{-1}(U_X).
		\]
		Then both restrictions
		\[
		\mu|_{W_U}:W_U\xrightarrow{\cong}U_Y,
		\qquad
		\nu|_{W_U}:W_U\xrightarrow{\cong}U_X
		\]
		are biholomorphic, and
		\[
		\nu|_{W_U}=\phi_U\circ\mu|_{W_U}.
		\]
		Set
		\[
		\tau:=\mu^*\sigma_Y\in H^0(W,\Omega_W^2).
		\]
		By \eqref{eq-form-comp}, we have
		\[
		\begin{aligned}
			(\nu|_{W_U})^*(\sigma_X|_{U_X})
			&=(\mu|_{W_U})^*
			\left(\phi_U^*(\sigma_X|_{U_X})\right) \\
			&=(\mu|_{W_U})^*(\sigma_Y|_{U_Y})
			=\tau|_{W_U}.
		\end{aligned}
		\]
		The open subset $W_U$ is dense in $W^\circ$. Thus the identity theorem,
		applied to the two holomorphic forms on $W^\circ$, gives
		\[
		\tau|_{W^\circ}
		=
		(\nu|_{W^\circ})^*(\sigma_X|_{X_{\rm reg}}).
		\]
		Thus $\tau$ is the extension to $W$ of the pull-back of
		$\sigma_X|_{X_{\rm reg}}$. Note that the proof of
		Claim \ref{claim-boundary-power-vanishing}  uses only the codimension-two issue of the smallness of
		$\phi$, the compatibility \eqref{eq-form-comp}, and the existence and
		functoriality of reflexive pull-back (\cite{KS21}) for $\sigma_X$, which is available
		because $X$ has terminal, and hence rational, singularities.    In particular, neither the Moishezonness of $Y$ nor the
		projectivity of $X$ enters the proof.  Consequently, the same argument
		gives 
		\begin{equation}\label{power-vani-nonmoishe}
			\left((\iota_E\circ r)^*\tau\right)^{n-1}=0
		\end{equation}
		for any reduced irreducible divisor $E\subset W\setminus W_U$, where
		$\iota_E:E\hookrightarrow W$ is the natural inclusion and
		$r:\widetilde E\to E$ is a resolution.

		\begin{step}
			The transform from $Y$ to $X$ and comparison of the BBF forms
		\end{step}
		
		Recall from \eqref{iso-pica} that there exists $L\in\operatorname{Pic}(Y)$
		such that $c_1(L)=v$. Define
		\[
		\mathscr L_X
		:=
		\left(\nu_*\mu^*L\right)^{**}.
		\]
		Then $\mathscr L_X$
		is a rank-one reflexive sheaf on $X$, and we have
		\begin{equation}\label{eq-transform-open}
			\mathscr L_X|_{U_X}
			\cong
			(\phi_U^{-1})^*(L|_{U_Y}).
		\end{equation}
		
		Since $X$ is strongly $\mathbb Q$-factorial, there exists an integer
		$m>0$ such that
		\[
		A
		:=
		\mathscr L_X^{[m]}
		:=
		\left(\mathscr L_X^{\otimes m}\right)^{**}
		\]
		is a holomorphic line bundle on $X$. Set
		\[
		L_m
		:=
		\left(\mu_*\nu^*A\right)^{**}.
		\]
		Then  $L_m$ is a rank-one reflexive sheaf on $Y$, and thus it
		is a holomorphic line bundle, because $Y$ is smooth. By
		\eqref{eq-transform-open},
		\[
		A|_{U_X}
		\cong
		(\phi_U^{-1})^*(L^{\otimes m}|_{U_Y}).
		\]
		Consequently, we have
		\[
		L_m|_{U_Y}
		\cong
		L^{\otimes m}|_{U_Y}.
		\]
		Since $Y\setminus U_Y$ has codimension at least two and both sides are
		reflexive sheaves on the normal space $Y$, this isomorphism extends
		uniquely to $Y$. Thus
		\begin{equation}\label{eq-return-transform}
			L_m\cong L^{\otimes m}.
		\end{equation}
		
		Note that the boundary
		power vanishing  \eqref{power-vani-nonmoishe} has been established. Then the argument  (neither the
		Moishezonness of $Y$ nor the projectivity of $X$ is used in that
		argument) proving \eqref{eq-step2-compare-X-to-Y}  applies in the present setting and gives
		\[
		q_Y(c_1(L_m))
		=
		\lambda q_X(c_1(A)),
		\qquad
		\lambda>0.
		\]
		By \eqref{eq-return-transform}, we have
		$c_1(L_m)=m v$. Thus we have
		\begin{equation}\label{eq-reverse-bbf}
			q_X(c_1(A))
			=
			\frac{m^2}{\lambda}q_Y(v)
			>0.
		\end{equation}

		\begin{step}
			Projectivity of $X$ and Moishezonness of $Y$
		\end{step}
		
		By the BBF-positivity \eqref{eq-reverse-bbf}, the projectivity criterion (\cite[Theorem 1.2]{BL22}) for the primitive symplectic
		variety $X$   implies that $X$ is projective.
		Thus 
		$Y$ is Moishezon.  This completes the proof of Theorem \ref{thm-positive-moishezon}.
	\end{proof}

	\subsection{A weak K\"ahler minimal-model condition $\left(\mathrm{KMM}_{0,d}\right)$}
	
	In this subsection, we show that the existence assumption in
	Theorem \ref{thm-positive-moishezon} concerning a small primitive
	symplectic model is satisfied for any hyperfujiki $4$-fold.
	More generally, we show that any hyperfujiki $2n$-fold admits such a
	small primitive symplectic model under a natural condition that is
	weaker than the validity of the $2n$-dimensional
	K\"ahler minimal model program.

	We first introduce a weaker K\"ahler minimal-model condition, which is inspired by Observation \ref{rem-kmm-relation-dim 4}.
	
	\begin{definition}\label{def-kmm-zero}
		
		Fix a positive integer $d$.  We say that
		\emph{condition $\left(\mathrm{KMM}_{0,d}\right)$ holds} (the subscript $0$ refers to $K_S$ being trivial) if the following
		assertion is true.  Let $V$ be any smooth connected compact K\"ahler $d$-fold
		for which there exist a smooth connected compact complex $d$-fold $S$ with
		$K_S\cong\mathcal O_S$  and a proper modification $\pi:V\to S$. \footnote{Note that in this setting,  $V$ satisfies that  $K_V \cong \mathcal{O}_V\left(E_\pi\right)$  for $E_\pi \geq 0$ exceptional and that both the Kodaira dimension and numerical dimension vanish: $\kappa(V)=\operatorname{nd}\left(K_V\right)=0$.}
		Then there exist a normal compact K\"ahler $d$-dimensional variety $V_{\rm{min}}$ with only
		terminal singularities, a bimeromorphic map
		\[
		\chi:V\dashrightarrow V_{\rm{min}},
		\]
		and a positive integer $\ell$ such that
		\[
		\omega_{V_{\rm{min}}}^{[\ell]}
		:=
		\left(\omega_{V_{\rm{min}}}^{\otimes\ell}\right)^{**}
		\]
		is a  nef (see \cite[Definition 2.1(2) and Remark 2.2]{DHP24} for the definition of nefness) invertible  sheaf.   
	\end{definition}
	
	Note that condition $\left(\mathrm{KMM}_{0,d}\right)$ only asserts the existence of the
	particular terminal nef model above.  It does not require the existence of an
	MMP sequence or semiampleness of the canonical bundle.  The following observation shows that condition 
	$\left(\mathrm{KMM}_{0,d}\right)$
	is  weaker than the usual standard K\"ahler MMP condition.

	\begin{observation}
		\label{rem-kmm-relation-dim 4}
		The known $4$-dimensional K\"ahler MMP (\cite[Theorem 1.1]{DHP24})
		implies that
		$\left(\mathrm{KMM}_{0,4}\right)$ holds. 
	\end{observation}
	
	\begin{proof}
		Let $V$ be as in
		Definition \ref{def-kmm-zero}, and let $\pi:V\to S$ be the
		corresponding proper modification, where $S$ is smooth and
		$K_S\cong\mathcal O_S$.  
		Since $V$ is a smooth
		compact K\"ahler fourfold, it is $\mathbb Q$-factorial and the pair
		$(V,0)$ is dlt.  Moreover, we have
		\[
		\omega_V
		\cong
		\pi^*\omega_S\otimes\mathcal O_V(R_\pi)
		\cong
		\mathcal O_V(R_\pi),
		\]
		where $R_\pi\geq 0$ is a $\pi$-exceptional divisor.  Thus  we  apply
		\cite[Theorem 1.1]{DHP24} to the pair $(V,B)=(V,0)$.  
		Consequently, $(V,0)$ admits a log minimal model, given by a
		bimeromorphic map $\chi:V\dashrightarrow V_{\rm{min}}$, where
		$V_{\rm{min}}$ is a normal compact K\"ahler $4$-dimensional variety.  
		Since
		$(V,0)$ is plt, \cite[Lemma 2.9(2)]{DHP24} shows that the divisor
		$E_{V_{\rm{min}}}$ occurring in the definition (\cite[Definition 2.8]{DHP24}) of a log minimal model
		vanishes, and $(V_{\rm{min}},0)$ is a log terminal model (\cite[Definition 2.8 (3)]{DHP24}) of $(V,0)$.

		\begin{claim}\label{claim-termi}
			$V_{\rm{min}}$ has only terminal singularities.
		\end{claim}
		
		\begin{proof}[Proof of Claim \ref{claim-termi}]
			Note that the variety $V_{\rm{min}}$ is
			$\mathbb Q$-factorial (\cite[Definition 2.8(1)-(2)]{DHP24}). By \cite[Definition 2.7(1)]{DHP24}, there is a
			positive integer $\ell$ such that
			\[
			\omega_{V_{\rm{min}}}^{[\ell]}
			:=
			\left(\omega_{V_{\rm{min}}}^{\otimes\ell}\right)^{**}
			\]
			is invertible.
			
			After eliminating the indeterminacies (e.g., \cite[Theorem 2.1.24]{MM07}) of $\chi$,  we obtain a smooth
			compact complex manifold $W$ and proper modifications
			\[
			\mu:W\to V,
			\qquad
			\nu:W\to V_{\rm{min}}
			\]
			such that $\chi\circ\mu=\nu$. Let
			\[
			F_1,\ldots,F_r
			\]
			be the prime divisors contained in the exceptional locus of $\nu$.
			Fix $1\leq i\leq r$.
			
			Suppose first that $F_i$ is $\mu$-exceptional. Since $V$ is smooth,
			we have $a(F_i,V,0)>0$. By \cite[Lemma 2.9(1)]{DHP24},
			\[
			a(F_i,V_{\rm{min}},0)
			\geq
			a(F_i,V,0)
			>0.
			\]
			
			Suppose now that $F_i$ is not $\mu$-exceptional. Then
			\[
			P_i:=\mu(F_i)_{\rm red}
			\]
			is a prime divisor on $V$, and $F_i$ is the strict transform of $P_i$
			on $W$. Since $F_i$ is $\nu$-exceptional, the divisor $P_i$ is
			contracted by $\chi$. Hence \cite[Definition 2.8(3)(iv)]{DHP24} gives
			\[
			a(P_i,V_{\rm{min}},0)
			>
			a(P_i,V,0)
			=0.
			\]
			Since $F_i$ is the strict transform of $P_i$ on the common resolution
			$W$, the definition of discrepancy gives
			\[
			a(F_i,V_{\rm{min}},0)
			=
			a(P_i,V_{\rm{min}},0)
			>0.
			\]
			Thus each prime divisor contained in the exceptional locus of $\nu$
			has positive discrepancy. Together with the invertibility of
			$\omega_{V_{\rm{min}}}^{[\ell]}$, this proves the claim  that $V_{\rm{min}}$ has
			only terminal singularities.
		\end{proof}

		Finally, \cite[Definition 2.8(3)(ii)]{DHP24} gives that
		$\omega_{V_{\rm{min}}}$  (note that, in
		\cite[Definition 2.7]{DHP24}, $K_{\bullet}$ denotes the canonical sheaf
		rather than a canonical divisor) is nef, where nefness is understood in the
		Bott--Chern sense of \cite[Definition 2.1(2) and Remark 2.2]{DHP24}.
		By the standard convention for
		the first Chern class of a $\mathbb Q$-line bundle, the Bott--Chern
		class determined by the canonical sheaf is $ \frac{1}{\ell}c_1(\omega_{V_{\rm{min}}}^{[\ell]}).$
		This class is nef, and thus its positive multiple $c_1(\omega_{V_{\rm{min}}}^{[\ell]})$ is nef.
		Therefore $\omega_{V_{\rm{min}}}^{[\ell]}$ is a nef invertible
		sheaf.
		Hence all the requirements in Definition \ref{def-kmm-zero} are
		fulfilled, and consequently $\left(\mathrm{KMM}_{0,4}\right)$ holds.
	\end{proof}

	\subsection{An equivalent characterization of Moishezonness via BBF positivity under condition $\left(\mathrm{KMM}_{0,d}\right)$}

	\begin{lemma}[{\cite[Lemma 1.3]{W21} Negativity Lemma}]\label{Negativity-Lemma}
		Let $h: Z \to Y$ be a proper modification between normal complex varieties. Let $B$ be a Cartier divisor on $Z$ such that $-B$ is $h$-nef. Then $B$ is effective if and only if $h_* B$ is effective.
	\end{lemma}

	We now show that a hyperfujiki
	manifold of dimension $2n$ has a small strongly $\mathbb Q$-factorial K\"ahler terminal model under the $\left(\mathrm{KMM}_{0,2n}\right)$ assumptions.
	
	\begin{lemma}
		\label{lem-sqf-model-KMM}
		Let $n\geq 2$, and assume that
		$\left(\mathrm{KMM}_{0,2n}\right)$ holds.  Let $Y$ be a hyperfujiki
		manifold of dimension $2n$.  Then there exist a normal compact K\"ahler
		strongly $\mathbb Q$-factorial  (Definition \ref{def-strong-Q-fac}) $2n$-fold $Y_{\rm{sqf}}$ with terminal
		singularities and a small bimeromorphic map
		\[
		\psi:Y\dashrightarrow Y_{\rm{sqf}},
		\]
		i.e.,  $\psi$ is an isomorphism in codimension one.
	\end{lemma}
	
	\begin{proof}
		Since $Y$ belongs to Fujiki class $\mathscr C$, there exists a smooth compact
		K\"ahler modification
		\[
		\pi:V\to Y.
		\]
		Then by the definition of a hyperfujiki manifold,  $V$ 
		satisfies the
		hypotheses in Definition \ref{def-kmm-zero}.  Since 
		$\left(\mathrm{KMM}_{0,2n}\right)$ holds,
		we obtain    a normal compact K\"ahler $2n$-fold
		$Y_{\rm{min}}$ with terminal singularities, a bimeromorphic map
		\[
		\chi:V\dashrightarrow Y_{\rm{min}},
		\]
		and a positive integer $\ell$ such that
		$\omega_{Y_{\rm{min}}}^{[\ell]}$ is a nef invertible sheaf.  Set
		\[
		\tau:=\chi\circ\pi^{-1}:Y\dashrightarrow Y_{\rm{min}}.
		\]

		We eliminate the indeterminacies (e.g., \cite[Theorem 2.1.24]{MM07}) of $\chi$,  and obtain a smooth compact K\"ahler (since $V$ is K\"ahler and the composition $\mu$ of finitely many blowups is a projective/ K\"ahler morphism) manifold
		$Z$, a proper modification $\mu:Z\to V$ (obtained as a composition of a  finite succession of blow-ups with smooth
		centers), and a proper modification
		$q:Z\to Y_{\rm{min}}$.  Set $p:=\pi\circ\mu$. Then $\tau\circ p=q$.  We now have  the following commutative diagram
		\[
		\begin{tikzcd}[column sep=large, row sep=large]
			&
			Z
			\arrow[dl, "p"']
			\arrow[d, "\mu"]
			\arrow[dr, "q"]
			&
			\\
			Y
			\arrow[rr, dashed, bend right=18, "\tau"']
			&
			V
			\arrow[l, "\pi"']
			\arrow[r, dashed, "\chi"]
			&
			Y_{\rm{min}}
		\end{tikzcd}
		\]

		\begin{claim}[{\cite[Proof of Theorem 3.8, Step 1]{BCDG25}}]\label{claim-small-BCDG}
			The bimeromorphic map $\tau$ is small.
		\end{claim}
		
		\begin{proof}[Proof of Claim \ref{claim-small-BCDG}]
			We now run exactly the same negativity-lemma argument as in \cite[Proof of Theorem 3.8, Step 1]{BCDG25} to prove that $\tau$ is small, merely giving more details.
			
			Choose a sufficiently divisible positive multiple $m$ of $\ell$.  Since
			$Y$ and $Y_{\rm{min}}$ have terminal singularities, we may write
			\[
			\omega_Z^{\otimes m}
			\cong
			p^*\omega_Y^{\otimes m}
			\otimes
			\mathcal O_Z\left(\sum_i ma_iE_i\right)
			\]
			and
			\[
			\omega_Z^{\otimes m}
			\cong
			q^*\omega_{Y_{\rm{min}}}^{[m]}
			\otimes
			\mathcal O_Z\left(\sum_j mb_jF_j\right),
			\]
			where the $E_i$ are the $p$-exceptional prime divisors,  the $F_j$ are the
			$q$-exceptional prime divisors, and $a_i,b_j>0$ for all indices $i,j$.
			Set
			\[
			D:=\sum_i ma_iE_i-\sum_j mb_jF_j.
			\]
			Since $\omega_Y\cong\mathcal O_Y$, the above isomorphisms on $\omega_Z^{\otimes m}$ give
			\[
			\mathcal O_Z(D)
			\cong
			q^*\omega_{Y_{\rm{min}}}^{[m]}.
			\]
			Then it follows  that $D$ is $p$-nef by the nefness of $\omega_{Y_{\rm{min}}}^{[m]}$, and $-D$ is $q$-nef, since $D$
			is $q$-numerically trivial.

			Since each $E_i$ is $p$-exceptional, we have
			\[
			p_*(-D)
			=
			p_*\left(\sum_j mb_jF_j\right)
			\geq 0.
			\]
			Applying Lemma \ref{Negativity-Lemma} to $h=p$ and $B=-D$, we obtain
			$-D\geq 0$, and thus $D\leq 0$.  Similarly, since each $F_j$ is
			$q$-exceptional, we have
			\[
			q_*D
			=
			q_*\left(\sum_i ma_iE_i\right)
			\geq 0.
			\]
			Applying Lemma \ref{Negativity-Lemma}  to $h=q$ and $B=D$, we obtain $D\geq 0$.
			Consequently, $D=0$.
			
			Since all the coefficients $a_i$ and $b_j$ are strictly positive, the
			$p$-exceptional and $q$-exceptional prime divisors on $Z$ coincide.  It
			follows that $\tau$ is an isomorphism in codimension one.  Indeed, if
			$\tau$ contracted a prime divisor on $Y$, its strict transform (by \cite[p. 215]{GR84}, the  center of $p$ has codimension at least $2$ and the strict transforms are well-defined) on $Z$
			would be $q$-exceptional but not $p$-exceptional; if $\tau^{-1}$
			contracted a prime divisor on $Y_{\rm{min}}$, its strict transform on $Z$
			would be $p$-exceptional but not $q$-exceptional, a contradiction.
			This completes the proof of claim \ref{claim-small-BCDG}. 
		\end{proof}

		Since $Y_{\rm{min}}$ has terminal singularities, the pair
		$(Y_{\rm{min}},0)$ is klt.  By the small strongly
		$\mathbb Q$-factorialization theorem \cite[Lemma 2.11]{CH24}, there
		exists a small projective modification
		\[
		\rho:Y_{\rm{sqf}}\to Y_{\rm{min}}
		\]
		such that $Y_{\rm{sqf}}$ is a normal strongly $\mathbb Q$-factorial (Definition \ref{def-strong-Q-fac})
		space.  Since $\rho$ is projective and $Y_{\rm{min}}$ is compact
		K\"ahler, the space $Y_{\rm{sqf}}$ is compact K\"ahler.   
		
		Since $\rho$ is small, there exist analytic subsets
		$B_{\rm{min}}\subset Y_{\rm{min}}$ and
		$B_{\rm{sqf}}\subset Y_{\rm{sqf}}$, both of codimension at least two,
		such that $\rho$ induces an isomorphism
		\[
		Y_{\rm{sqf}}\setminus B_{\rm{sqf}}
		\cong
		Y_{\rm{min}}\setminus B_{\rm{min}}.
		\]
		Let
		\[
		\psi:=\rho^{-1}\circ\tau:Y\dashrightarrow Y_{\rm{sqf}}.
		\]
		Then $\psi$ is bimeromorphic and is an isomorphism in codimension one.
		It remains to prove the following claim.

		\begin{claim}\label{claim-termi-sin}
			$Y_{\rm{sqf}}$ has only terminal singularities.
		\end{claim}
		
		\begin{proof}[Proof of Claim \ref{claim-termi-sin}]
			
			Take a
			positive integer $r$ such that both 
			$\omega_{Y_{\rm{sqf}}}^{[r]}$ and   $\omega_{Y_{\rm{min}}}^{[r]}$ 
			are
			invertible sheaves. 
			
			Since the restriction
			\[
			\rho:Y_{\rm{sqf}}\setminus B_{\rm{sqf}}
			\to
			Y_{\rm{min}}\setminus B_{\rm{min}}
			\]
			is an isomorphism, the reflexive pluricanonical sheaves
			$\omega_{Y_{\rm{sqf}}}^{[r]}$ and
			$\rho^*\omega_{Y_{\rm{min}}}^{[r]}$ are isomorphic on
			$Y_{\rm{sqf}}\setminus B_{\rm{sqf}}$.  Both sheaves are invertible, and
			$B_{\rm{sqf}}$ has codimension at least two.  Thus this isomorphism
			extends uniquely to $Y_{\rm{sqf}}$, namely
			\begin{equation}\label{eq-rho}
				\omega_{Y_{\rm{sqf}}}^{[r]}
				\cong
				\rho^*\omega_{Y_{\rm{min}}}^{[r]}.
			\end{equation}
			
			Let
			\[
			a:W\to Y_{\rm{sqf}}
			\]
			be a resolution, where $W$ is a complex manifold, and set
			\[
			b:=\rho\circ a:W\to Y_{\rm{min}}.
			\]
			Then $b$ is a resolution of $Y_{\rm{min}}$.  By the
			resolution-independence of terminality, the terminality of
			$Y_{\rm{min}}$ may be tested via $b$.  Thus
			\begin{equation}\label{eq-ymin}
				\omega_W^{\otimes r}
				\cong
				b^*\omega_{Y_{\rm{min}}}^{[r]}
				\otimes
				\mathcal O_W\left(\sum_{i=1}^N r\beta_iG_i\right),
			\end{equation}
			where $G_1,\ldots,G_N$ are the $b$-exceptional prime divisors on $W$,
			and the numbers $\beta_i\in\mathbb Q$ satisfy
			$r\beta_i\in\mathbb Z$ and
			\begin{equation}\label{eq-pos}
				\beta_i>0
				\qquad
				\text{for each }1\leq i\leq N.
			\end{equation}
			
			The smallness of $\rho$ implies that the $a$-exceptional and
			$b$-exceptional prime divisors on $W$ coincide.  Indeed, if
			$G\subset W$ is $a$-exceptional, then $a(G)$ has codimension at least
			two in $Y_{\rm{sqf}}$, so $b(G)=\rho(a(G))$ has codimension at least two
			in $Y_{\rm{min}}$.  Thus $G$ is $b$-exceptional.  Conversely, if $G$
			is $b$-exceptional but not $a$-exceptional, then $a(G)$ is a prime
			divisor in $Y_{\rm{sqf}}$ whose image under $\rho$ has codimension at
			least two in $Y_{\rm{min}}$, contradicting the smallness of $\rho$.
			
			Using \eqref{eq-rho} and $b=\rho\circ a$, we get
			\[
			a^*\omega_{Y_{\rm{sqf}}}^{[r]}
			\cong
			a^*\rho^*\omega_{Y_{\rm{min}}}^{[r]}
			\cong
			b^*\omega_{Y_{\rm{min}}}^{[r]}.
			\]
			Substituting this into \eqref{eq-ymin}, we obtain
			\[
			\omega_W^{\otimes r}
			\cong
			a^*\omega_{Y_{\rm{sqf}}}^{[r]}
			\otimes
			\mathcal O_W\left(\sum_{i=1}^N r\beta_iG_i\right).
			\]
			By \eqref{eq-pos}, $Y_{\rm{sqf}}$ has only terminal singularities.
			This completes the proof of Claim \ref{claim-termi-sin} and hence the proof of  Lemma \ref{lem-sqf-model-KMM}.
		\end{proof}
	\end{proof}

	Lemma \ref{lem-sqf-model-KMM} provides a small strongly $\mathbb Q$-factorial
	terminal K\"ahler model, under the $\left(\mathrm{KMM}_{0,2n}\right)$ assumptions.
	We now show that,  the 
	symplectic form descends to this model and endows it with the structure
	of a primitive symplectic variety.

	\begin{proposition}\label{prop-kmm-model}
		Let $n\geq 2$, and assume that $\left(\mathrm{KMM}_{0,2n}\right)$ holds.
		Let $(Y,\sigma_Y)$ be a hyperfujiki manifold of dimension $2n$.  Then there
		exist a normal strongly $\mathbb Q$-factorial compact K\"ahler $2n$-fold $X$
		with only terminal singularities, a small bimeromorphic map
		\[
		\phi:Y\dashrightarrow X,
		\]
		and a reflexive two-form
		\[
		\sigma_X
		\in
		H^0\left(X,\left(\wedge^2\Omega_X^1\right)^{**}\right)
		\]
		satisfying 
		\begin{equation}
			\label{eq-kmm-form}
			\sigma_Y|_{U_Y}
			=
			\phi_U^*\left(\sigma_X|_{U_X}\right), 
		\end{equation} 
		where  $U_Y\subset Y$ and $U_X\subset X$ are 
		the maximal analytic Zariski open subsets on which $\phi$ induces a
		biholomorphism $\phi_U:U_Y\xrightarrow{\cong}U_X.$
		Moreover,
		\[
		\left(X,\sigma_X|_{X_{\rm{reg}}}\right)
		\]
		is a primitive symplectic variety in the sense of
		Definition \ref{def-symp-varie}.
		
		As a result, these conclusions
		hold for any $4$-dimensional hyperfujiki manifold, based on Observation \ref{rem-kmm-relation-dim 4}.
		
	\end{proposition}

	\begin{proof}
		Applying Lemma \ref{lem-sqf-model-KMM} to $Y$, we obtain a normal compact
		K\"ahler strongly $\mathbb Q$-factorial $2n$-fold $Y_{\rm{sqf}}$ with only
		terminal singularities and a small bimeromorphic map
		\[
		\psi:Y\dashrightarrow Y_{\rm{sqf}}
		\]
		which is an isomorphism in codimension one. Set
		\[
		X:=Y_{\rm{sqf}},
		\qquad
		\phi:=\psi.
		\]
		
		Let $U_Y\subset Y$ and $U_X\subset X$ be the maximal analytic Zariski open
		subsets on which $\phi$ induces a biholomorphism
		\[
		\phi_U:U_Y\xrightarrow{\cong}U_X.
		\]
		Since $\phi$ is small, we have
		\begin{equation}\label{codim-geq-2}
			\operatorname{codim}_Y(Y\setminus U_Y)\geq 2,
			\qquad
			\operatorname{codim}_X(X\setminus U_X)\geq 2.
		\end{equation}
		Since $Y$ is smooth and $U_Y\cong U_X$, the open subset $U_X$ is smooth. In
		particular, $U_X\subset X_{\rm{reg}}$.
		
		Set 
		\[
		s_U:=(\phi_U^{-1})^*(\sigma_Y|_{U_Y})
		\in H^0(U_X,\Omega^2_{U_X}).
		\]
		Since $U_X\subset X_{\rm{reg}}$, we have
		\[
		\left(\wedge^2\Omega_X^1\right)^{**}\big|_{U_X}
		\cong
		\Omega^2_{U_X}.
		\]
		Let $i:U_X\hookrightarrow X$ be the natural inclusion. Since $X$ is normal and
		$X\setminus U_X$ has codimension at least $2$, Serre's extension theorem
		(e.g., \cite[Proposition 5.29]{PR94}) for reflexive sheaves gives
		\[
		\left(\wedge^2\Omega_X^1\right)^{**}
		\cong
		i_*\left(
		\left(\wedge^2\Omega_X^1\right)^{**}\big|_{U_X}
		\right).
		\]
		Thus $s_U$ extends uniquely to a section
		\begin{equation}\label{def-sigma-X}
			\sigma_X\in
			H^0\left(X,\left(\wedge^2\Omega_X^1\right)^{**}\right).
		\end{equation}
		By construction, $\sigma_X|_{U_X}=s_U$, and therefore
		\[
		\sigma_Y|_{U_Y}
		=
		\phi_U^*\left(\sigma_X|_{U_X}\right).
		\]
		This proves \eqref{eq-kmm-form}.
		
		We now fix a common resolution which will be used in the rest of the proof.
		Let $\Gamma$ be the graph of $\phi$, endowed with the reduced structure, and
		let
		\[
		p:\Gamma\to Y,
		\qquad
		q:\Gamma\to X
		\]
		be the natural projections. Choose a resolution
		\[
		\rho:W\to\Gamma
		\]
		which is biholomorphic over $\Gamma_{\rm{reg}}$
		(e.g., \cite[Theorem 5.4.2]{AHV18}). Set
		\[
		\mu:=p\circ\rho:W\to Y,
		\qquad
		\nu:=q\circ\rho:W\to X.
		\]
		Then $\nu:W\to X$ is a resolution of $X$. We have the diagram
		\[
		\begin{tikzcd}
			& W \arrow[d,"\rho"]
			\arrow[ddl,bend right=15,"\mu"']
			\arrow[ddr,bend left=15,"\nu"] & \\
			& \Gamma \arrow[dl,"p"'] \arrow[dr,"q"] & \\
			Y \arrow[rr,dashed,"\phi"'] & & X .
		\end{tikzcd}
		\]
		Define
		\begin{equation}\label{def-Gamma-U}
			\Gamma_U:=p^{-1}(U_Y)=q^{-1}(U_X),
			\qquad
			W^\circ:=\nu^{-1}(X_{\rm{reg}}),
			\qquad
			W_U:=\rho^{-1}(\Gamma_U)=\mu^{-1}(U_Y)=\nu^{-1}(U_X).
		\end{equation}
		Then $\Gamma_U$ is the graph of the biholomorphism
		$\phi_U:U_Y\xrightarrow{\cong}U_X$. In particular, $\Gamma_U\cong U_Y$, so
		$\Gamma_U\subset\Gamma_{\rm{reg}}$. Since $\rho$ is biholomorphic over
		$\Gamma_{\rm{reg}}$, it is biholomorphic over $\Gamma_U$. Thus
		\[
		\mu|_{W_U}:W_U\xrightarrow{\cong}U_Y,
		\qquad
		\nu|_{W_U}:W_U\xrightarrow{\cong}U_X.
		\]
		For proving Proposition \ref{prop-kmm-model}, it suffices to prove the following claim.
		
		\begin{claim}\label{claim-pri-symp}
			The pair $\left(X,\sigma_X|_{X_{\rm{reg}}}\right)$ is a primitive symplectic
			variety.
		\end{claim}
		
		\begin{proof}[Proof of Claim \ref{claim-pri-symp}]
			It suffices to check the conditions in Definition \ref{def-symp-varie}. 
			First, we prove that $\sigma_X|_{X_{\rm{reg}}}$ is non-degenerate. Put
			\[
			\alpha:=\sigma_X|_{X_{\rm{reg}}}
			\in H^0(X_{\rm{reg}},\Omega^2_{X_{\rm{reg}}}).
			\]
			Then
			\[
			\alpha^{\wedge n}
			\in
			H^0(X_{\rm{reg}},\omega_{X_{\rm{reg}}}).
			\]
			On $U_X$, the form $\alpha^{\wedge n}$ corresponds under
			$\phi_U:U_Y\xrightarrow{\cong}U_X$ to $\left(\sigma_Y|_{U_Y}\right)^{\wedge n}.$
			Since $\sigma_Y$ is non-degenerate on $Y$, this section is nowhere vanishing on
			$U_X$. Therefore the zero locus of $\alpha^{\wedge n}$ is contained in
			$X_{\rm{reg}}\setminus U_X$. By \eqref{codim-geq-2}, this subset has
			codimension at least $2$ in $X_{\rm{reg}}$. On the other hand, the zero locus
			of a nonzero holomorphic section of a line bundle on a complex manifold is
			divisorial if it is non-empty. Thus the zero locus is empty. Therefore
			$\alpha=\sigma_X|_{X_{\rm{reg}}}$ is non-degenerate everywhere on
			$X_{\rm{reg}}$.
			
			Next, we prove that $\alpha$ is $d$-closed. By construction,
			\[
			\alpha|_{U_X}=s_U,
			\qquad
			\phi_U^*s_U=\sigma_Y|_{U_Y}.
			\]
			Since $Y$ satisfies the $\partial\overline{\partial}$-lemma, its holomorphic
			two-form $\sigma_Y$ is $d$-closed. Indeed, $\partial\sigma_Y$ is
			$\partial$-exact and $\overline{\partial}$-closed, so the
			$\partial\overline{\partial}$-lemma implies that it is
			$\partial\overline{\partial}$-exact; its bidegree $(3,0)$ then forces
			$\partial\sigma_Y=0$. Thus $d\alpha|_{U_X}=0$. Since $\alpha$ is a
			holomorphic two-form on $X_{\rm{reg}}$, the form $d\alpha$ is a holomorphic
			three-form on $X_{\rm{reg}}$. The open subset $U_X$ is dense in
			$X_{\rm{reg}}$, so the identity theorem gives the closedness $ d\alpha=0$
			on $X_{\rm{reg}}$.
			
			We now verify the extension condition. The pull-back
			\[
			\eta:=(\nu|_{W^\circ})^*(\sigma_X|_{X_{\rm{reg}}})
			\]
			is a holomorphic two-form on $W^\circ$. On $W_U$, the morphisms satisfy
			$\nu=\phi_U\circ\mu$. Therefore, by the construction of $s_U$,
			\begin{equation}\label{exten-sympl}
				\eta|_{W_U}
				=
				(\mu^*\sigma_Y)|_{W_U}.
			\end{equation}
			Since $W_U$ is dense in $W^\circ$, the identity theorem gives
			\[
			\eta=(\mu^*\sigma_Y)|_{W^\circ}.
			\]
			But $\mu^*\sigma_Y\in H^0(W,\Omega_W^2)$ is a holomorphic two-form on the
			complex manifold $W$. Thus $\mu^*\sigma_Y$ extends $\eta$ from $W^\circ$ to $W$.
			This verifies the extension condition in Definition \ref{def-symp-varie}.
			
			It remains to check primitiveness. Since
			$\operatorname{codim}_{X_{\rm{reg}}}(X_{\rm{reg}}\setminus U_X)\geq 2$ and
			$X_{\rm{reg}}$ is smooth, Hartogs extension for sections of vector bundles
			gives
			\[
			H^0(X_{\rm{reg}},\Omega^2_{X_{\rm{reg}}})
			\cong
			H^0(U_X,\Omega^2_{U_X}).
			\]
			Similarly, since $Y$ is smooth and
			$\operatorname{codim}_Y(Y\setminus U_Y)\geq 2$, we have
			\[
			H^0(Y,\Omega^2_Y)
			\cong
			H^0(U_Y,\Omega^2_{U_Y}).
			\]
			Using $U_Y\cong U_X$, we obtain
			\[
			\begin{aligned}
				H^0(X_{\rm{reg}},\Omega^2_{X_{\rm{reg}}})
				&\cong
				H^0(U_X,\Omega^2_{U_X})  \\
				&\cong
				H^0(U_Y,\Omega^2_{U_Y})  \\
				&\cong
				H^0(Y,\Omega^2_Y)
				=
				\mathbb C\sigma_Y.
			\end{aligned}
			\]
			Under the above composed isomorphism, $\sigma_X|_{X_{\rm{reg}}}$ is mapped  to
			$\sigma_Y$. Thus we have
			\[
			H^0(X_{\rm{reg}},\Omega^2_{X_{\rm{reg}}})
			=
			\mathbb C(\sigma_X|_{X_{\rm{reg}}}).
			\]
			
			Finally, we prove that $H^1(X,\mathcal O_X)=0$. Since $X$ has terminal
			singularities, it has rational singularities. Since $Y$ also has rational singularities, the
			morphisms $\nu:W\to X$ and $\mu:W\to Y$ satisfy
			\[
			\nu_*\mathcal O_W=\mathcal O_X,
			\qquad
			R^i\nu_*\mathcal O_W=0
			\quad
			\text{for each } i>0,
			\]
			and
			\[
			\mu_*\mathcal O_W=\mathcal O_Y,
			\qquad
			R^i\mu_*\mathcal O_W=0
			\quad
			\text{for each } i>0.
			\]
			It then follows from the Leray spectral sequence that
			\[
			H^1(X,\mathcal O_X)
			\cong
			H^1(W,\mathcal O_W)
			\cong
			H^1(Y,\mathcal O_Y).
			\]
			Since $Y$ is simply connected, $H^1(Y,\mathbb C)=0$. Since $Y$ satisfies the
			$\partial\overline{\partial}$-lemma, this implies $H^1(Y,\mathcal O_Y)=0$.
			Thus $H^1(X,\mathcal O_X)=0$.
			Therefore $\left(X,\sigma_X|_{X_{\rm{reg}}}\right)$ is a primitive symplectic
			variety in the sense of Definition \ref{def-symp-varie}. This completes the
			proof of Claim \ref{claim-pri-symp}.
		\end{proof}

		By construction, $X$ is a strongly $\mathbb Q$-factorial primitive symplectic
		K\"ahler variety with only terminal singularities,  and
		$\phi:Y\dashrightarrow X$ is small. Furthermore, $\sigma_X|_{X_{\rm{reg}}}$ satisfies the condition \eqref{eq-kmm-form}.
		This completes the proof of Proposition \ref{prop-kmm-model}.
	\end{proof}

	As a direct consequence of Theorem \ref{thm-positive-moishezon}, Proposition \ref{prop-kmm-model} and Theorem \ref{thm-arb-positive}, we obtain the following corollary. 
	
	\begin{corollary}
		\label{cor-hf-moishezon}
		Let $n\geq 2$, and assume that $\left(\mathrm{KMM}_{0,2n}\right)$ holds.
		Let $(Y,\sigma_Y)$ be a hyperfujiki manifold of dimension $2n$. Then we have the following
		\[
		Y\text{ is Moishezon}
		\Longleftrightarrow
		\exists v\in H^2(Y,\mathbb{Z})\cap H^{1,1}(Y)
		\text{ such that }
		q_Y(v)>0.
		\]
	\end{corollary}

	In particular, together with Observation \ref{rem-kmm-relation-dim 4},  we have the following corollary.
	
	\begin{corollary}
		\label{cor-hf-moishezon-4dim}
		Let $(Y,\sigma_Y)$ be a hyperfujiki $4$-fold. Then
		\[
		Y\text{ is Moishezon}
		\Longleftrightarrow
		\exists v\in H^2(Y,\mathbb{Z})\cap H^{1,1}(Y)
		\text{ such that }
		q_Y(v)>0.
		\]
	\end{corollary}

	\section{Application to 
		the  description of the Moishezon locus}\label{sec-main-hf}

	As an application of the BBF-positivity characterization of Moishezonness (and projectiveness), combined with the period theory developed in \cite{ACRT18}, we obtain the following description of the Moishezon locus for certain smooth families. Note that Koll\'ar deals with 
	more general smooth families (\cite[Theorem 21]{Kol22}). But for families with each fiber being hyperfujiki or hyperk\"ahler, it
	only yields that either $\operatorname{Moi}(f)=S$ or
	$\operatorname{Moi}(f)$ is contained in a countable union of proper 
	analytic subsets.\footnote{Although $R^{2}f_{*}\mathcal{O}_{\mathcal{X}}$ is a holomorphic line bundle in this setting, only the zero loci of the sections produced by the exponential sequence are automatically hypersurfaces. The very-big loci occurring in Koll\'ar's argument are controlled only by proper analytic subsets. Furthermore, the argument of \cite[Theorem 21]{Kol22} yields only the inclusion ``is contained in", rather than an equality.}
	Note that when the base $S$ is $1$-dimensional (so proper analytic subsets of $S$ are exactly discrete subsets), this result is also essentially contained
	in \cite[Theorem 1.0.2, Proposition 4.0.8]{Ba15}, \cite[Theorem 1.4]{RT21} and \cite[Theorem 1.1]{LRWW24}.
	Clearly, this type of result readily yields results on degenerations of Moishezon or projective manifolds.

	\begin{theorem}\label{thm-intro-hf-copy2}
		Let $f:\mathcal{X}\to S$ be a proper holomorphic submersion from a complex
		manifold $\mathcal{X}$ to a connected simply connected complex manifold $S$, with
		connected fibers. Assume that one of the following two conditions holds:
		\begin{enumerate}
			\item 
			either each fiber $X_t:=f^{-1}(t)$ is a hyperfujiki $4$-fold, or each fiber $X_t$ is a hyperfujiki $2n$-fold and $\left(\mathrm{KMM}_{0,2n}\right)$ holds ($n\geq 3$);
			\item every fiber $X_t$ is a hyperk\"ahler manifold.
		\end{enumerate}
		Then the Moishezon locus, which in the hyperk\"ahler case coincides with the
		projective locus,
		\[
		\operatorname{Moi}(f):=\{\,t\in S\mid X_t\text{ is Moishezon}\,\},
		\]
		is either  $S$ or    at most  (possibly empty) a countable
		union of hypersurfaces of $S$.
	\end{theorem}

	\begin{remark}
		In the hyperk\"ahler case, it is classical that, for a sufficiently small
		representative of the Kuranishi family of a (projective) hyperk\"ahler manifold $X$, the
		locus in the base parametrizing projective fibers is a countable union of
		hypersurfaces (e.g., \cite[Remark 5.4]{GLR13}).
	\end{remark}
	
	Since Theorem \ref{thm-arb-positive} holds in arbitrary even dimension, the inclusion \eqref{moi-subset-Dv} in the proof of Theorem \ref{thm-intro-hf-copy2} remains valid for any family satisfying the assumptions of the following corollary. This yields the following, somewhat coarser, description of the Moishezon locus.
	
	\begin{corollary}
		Let $f:\mathcal X\to S$ be a proper holomorphic submersion from a
		complex manifold $\mathcal X$ onto a connected simply connected
		complex manifold $S$, with connected fibers. Assume, moreover, that
		each fiber $X_t$ is a hyperfujiki $2n$-fold. Then
		\[
		\operatorname{Moi}(f)
		\subseteq
		\bigcup_{\substack{v\in\Lambda\\ q(v)>0}}D_v,
		\]
		where $D_v$ is as in \eqref{def-Dv}. Consequently, unless $D_v=S$ for
		some $v\in\Lambda$ with $q(v)>0$, the Moishezon locus
		$\operatorname{Moi}(f)$ is contained in a union of at most countably
		many  hypersurfaces of $S$.
	\end{corollary}

	\begin{proof}[Proof of Theorem \ref{thm-intro-hf-copy2}]
		\setcounter{proofstep}{0}
		Since $S$ is simply connected, the local system $R^2f_*\mathbb{Z}$ is globally
		trivial. Fix a point $s_0\in S$, and set
		\[
		\Lambda:=H^2(X_{s_0},\mathbb{Z}).
		\]
		We identify $R^2f_*\mathbb{Z}\cong\Lambda_S$, where $\Lambda_S$ denotes the
		constant local system with fiber $\Lambda$. Under this identification, each
		$v\in\Lambda$ determines a constant section of $R^2f_*\mathbb{Z}$, still
		denoted by $v$. Its value at $t\in S$ is the Gauss--Manin parallel transport
		\[
		v_t\in H^2(X_t,\mathbb{Z})
		\]
		of $v$.

		Under the fixed global Gauss--Manin trivialization
		\[
		\tau:
		R^2f_*\mathbb{C}\otimes_{\mathbb{C}}\mathcal{O}_S
		\cong
		\Lambda_{\mathbb{C}}\otimes_{\mathbb{C}}\mathcal{O}_S,
		\qquad
		\Lambda_{\mathbb{C}}:=\Lambda\otimes_{\mathbb{Z}}\mathbb{C},
		\]
		let
		\[
		\tau_t:H^2(X_t,\mathbb{C})\longrightarrow\Lambda_{\mathbb{C}}
		\]
		be the induced isomorphism on the fiber over $t\in S$. Thus
		$\tau_t(v_t)=v$ for each $v\in\Lambda$. Choose a BBF form $q$ of the
		reference fiber $X_{s_0}$ and regard it as a quadratic form on
		$\Lambda_{\mathbb{R}}:=\Lambda\otimes_{\mathbb{Z}}\mathbb{R}$. We use the
		same symbol $q$ for its complexification to $\Lambda_{\mathbb{C}}$. Let $B_0$
		denote its symmetric complex-bilinear polarization on $\Lambda_{\mathbb{C}}$,
		namely,
		\[
		B_0(\alpha,\beta)
		:=
		\frac{1}{2}
		\bigl(
		q(\alpha+\beta)-q(\alpha)-q(\beta)
		\bigr),
		\qquad
		\alpha,\beta\in\Lambda_{\mathbb{C}}.
		\]
		For each $t\in S$, let $q_{X_t}$ be a BBF form of $X_t$, and let $B_t$
		denote its complex-bilinear polarization.

		Recall from Definition \ref{def-bbf-fourfold} that the BBF form is defined over $\mathbb{R}$, i.e., $q_{Z,\sigma}(\alpha)\in\mathbb{R}$ for any real class $\alpha\in H^2(Z,\mathbb{R})$.
		Then neither the BBF-positivity condition nor the associated
		orthogonality equation \eqref{equi-1-1type-orthogo} depends on the auxiliary positive constants $C_{X_t}$
		in Definition \ref{def-bbf-fourfold}. Since only the sign of the BBF square
		and the corresponding orthogonality equation are used, the normalization of
		the BBF forms will not matter.

		We now give  the precise comparison between
		$B_t$ and  $B_0$.
		
		\begin{claim}\label{BBF-posi-preserve}
			For each $t\in S$, there exists $c_t\in\mathbb{R}_{>0}$ such that
			\[
			B_t(\alpha,\beta)
			=
			c_tB_0\bigl(\tau_t(\alpha),\tau_t(\beta)\bigr)
			\]
			for each $\alpha,\beta\in H^2(X_t,\mathbb{C})$. Consequently, for each
			$v\in\Lambda$ and each $t\in S$,
			\[
			q(v)>0
			\quad\Longleftrightarrow\quad
			q_{X_t}(v_t)>0.
			\]
		\end{claim}
		
		\begin{proof}[Proof of Claim \ref{BBF-posi-preserve}]
			In both cases, each fiber is hyperfujiki and thus a simple
			$\partial\bar\partial$-complex symplectic manifold
			(\cite[Definition 1.1]{ACRT18}). Fix a point $a\in S$. After shrinking to a
			sufficiently small neighborhood $U$ of $a$, the restricted family over $U$ is
			induced from the Kuranishi universal deformation of $X_a$, which is smooth and
			universal by \cite[Corollary 3.5]{ACRT18}.
			
			After shrinking $U$ further if necessary, a holomorphic symplectic form on
			$X_a$ extends to a local relative holomorphic symplectic form by
			\cite[Proposition 3.1]{ACRT18}. Let $\Omega_U$ be such a relative form, and
			write
			\[
			\sigma_t:=\Omega_U|_{X_t}.
			\]
			By \cite[Corollary 3.5 and Corollary 4.7]{ACRT18}, under the Gauss--Manin
			identification of $H^2(X_t,\mathbb{C})$ with $H^2(X_a,\mathbb{C})$, the
			quadratic forms $q_{\sigma_a}$ and $q_{\sigma_t}$ define the same (projective) quadric, in
			the sense that they differ by a nonzero scalar:
			\begin{equation}\label{local-com-posi}
				q_{\sigma_t}=\lambda_U(t)q_{\sigma_a}
				\quad\text{on }H^2(X_a,\mathbb{C})
			\end{equation}
			for some $\lambda_U(t)\in\mathbb{C}^*$.\footnote{Note that in
				\cite[Corollary 4.7]{ACRT18}, the phrase ``define the same quadric'' means
				that the corresponding quadratic forms agree up to multiplication by a
				nonzero scalar. This is clear from the proof of
				\cite[Corollary 4.7]{ACRT18}, and is also precisely the sense in which
				\cite[Proposition 4.8]{ACRT18} invokes \cite[Corollary 4.7]{ACRT18}.} Applying \cite[Corollary 4.4]{ACRT18} to $\sigma_t+\bar{\sigma}_t$ (possibly after shrinking $U$ again), we obtain $\lambda_U(t)>0$.

			For each $t\in S$, choose a path $\gamma$
			from $s_0$ to $t$. The compact set $\gamma([0,1])$ is covered by finitely many
			neighborhoods on which the  local comparison \eqref{local-com-posi} holds. Thus, after
			subdividing the path, the BBF form of $X_t$,
			transported to $\Lambda_{\mathbb{R}}$, is obtained from $q$ by multiplying
			finitely many positive real numbers. Therefore there exists
			$c_t\in\mathbb{R}_{>0}$ such that the transported
			BBF form of $X_t$ is $c_tq$. Polarizing this equality
			gives
			\[
			B_t(\alpha,\beta)
			=
			c_tB_0\bigl(\tau_t(\alpha),\tau_t(\beta)\bigr)
			\]
			for each $\alpha,\beta\in H^2(X_t,\mathbb{C})$. The asserted equivalence
			of positivity follows immediately.  This completes the proof of Claim \ref{BBF-posi-preserve}.
		\end{proof}
		
		Recall that  the $2$-integral cohomology group of a hyperfujiki manifold is torsion free. For notational simplicity, we  do not
		distinguish an integral class from its images in real and complex cohomology.
		For each $v\in\Lambda$, set
		\begin{equation}\label{def-Dv}
			D_v:=\{\,t\in S\mid v_t\in H^{1,1}(X_t)\,\}.
		\end{equation}

		Recall that for a simple complex symplectic manifold $(Y,\sigma)$ satisfying the
		$\partial\bar\partial$-lemma, the following BBF-orthogonal Hodge decomposition
		holds (\cite[Propositions 2.8, 2.10]{CT18}):
		\[
		H^2(Y,\mathbb{C})
		=
		\operatorname{Span}\{[\sigma],[\bar{\sigma}]\}
		\oplus
		H_{\bar{\partial}}^{1,1}(Y).
		\]
		Let $B_{Y,\sigma}$ be the symmetric complex-bilinear polarization of a BBF
		form $q_{Y,\sigma}$.
		Then, for a real class $\alpha\in H^2(Y,\mathbb{R})$, one has
		\begin{equation}\label{equi-1-1type-orthogo}
			\alpha\in H^{1,1}(Y)
			\Longleftrightarrow
			B_{Y,\sigma}(\alpha,\sigma)=0,  
		\end{equation}
		because
		$B_{Y,\sigma}(\alpha,\bar{\sigma})
		=\overline{B_{Y,\sigma}(\alpha,\sigma)}$.
		Consequently,
		\[
		D_v
		=
		\{\,t\in S\mid B_t(v_t,\sigma_t)=0\,\},
		\]
		where $\sigma_t$ is any generator of $H^{2,0}(X_t)$.

		\begin{step}\label{step-hodge-locus}
			$D_v$ is an analytic subset for each $v\in\Lambda$
		\end{step}
		
		Fix $v\in\Lambda$. For any point $a\in S$, there exists an open neighborhood
		$U_a\subseteq S$ of $a$ such that the restricted family
		$f^{-1}(U_a)\to U_a$ is induced from the universal deformation of $X_a$.
		Since $X_a$ is a simple $\partial\bar\partial$-complex symplectic manifold,
		\cite[Corollary 3.5 and Theorem 4.1]{ACRT18} shows that, under the
		Gauss--Manin trivialization, the local period map
		\[
		\mathcal{P}_{U_a}:
		U_a\longrightarrow
		\operatorname{Grass}\bigl(1,\Lambda_{\mathbb{C}}\bigr),
		\qquad
		t\longmapsto H^{2,0}(X_t),
		\]
		is holomorphic. Let $\mathscr{U}$ denote the tautological line bundle on this
		Grassmannian. Then
		\[
		F^2:=\mathcal{P}_{U_a}^*\mathscr{U}
		\subset
		\Lambda_{\mathbb{C}}\otimes_{\mathbb{C}}\mathcal{O}_{U_a}
		\cong
		\left(
		R^2f_*\mathbb{C}\otimes_{\mathbb{C}}\mathcal{O}_S
		\right)\big|_{U_a}
		\]
		is a holomorphic line subbundle\footnote{This also follows by an argument
			similar to that in \cite[\S 10.2.1]{Voi02}, since
			\cite[Theorem 10.10]{Voi02} applies in the present setting. In the
			hyperk\"ahler case, the assertion follows directly from
			\cite[\S 10.2.1]{Voi02}.} whose fiber at $t\in U_a$ is
		$F_t^2=H^{2,0}(X_t)$.
		
		Let $V\subseteq U_a$ be a trivialization domain for $F^2$, and choose a local
		holomorphic frame $\sigma_V$ of $F^2|_V$. Then
		$H^{2,0}(X_t)=\mathbb{C}\sigma_V(t)$ for each $t\in V$. By Claim
		\ref{BBF-posi-preserve}, for each $t\in V$ one has
		\[
		\begin{aligned}
			v_t\in H^{1,1}(X_t)
			&\Longleftrightarrow
			B_t\bigl(v_t,\tau_t^{-1}(\sigma_V(t))\bigr)=0\\
			&\Longleftrightarrow
			c_tB_0\bigl(v,\sigma_V(t)\bigr)=0\\
			&\Longleftrightarrow
			B_0\bigl(v,\sigma_V(t)\bigr)=0,
		\end{aligned}
		\]
		where the second equivalence uses $\tau_t(v_t)=v$. Since
		$v\in\Lambda_{\mathbb{C}}$ is fixed, the linear functional
		$B_0(v,-)\in\Lambda_{\mathbb{C}}^{\vee}$ is fixed. Since $\sigma_V$ is
		holomorphic, the function
		\[
		g_{v,V}(t):=B_0\bigl(v,\sigma_V(t)\bigr)
		\]
		is holomorphic on $V$. Therefore $D_v\cap V$ is the zero locus of the
		holomorphic function $g_{v,V}$.
		
		Since the trivialization domains $V$ cover $U_a$ and closedness is a local
		property, $D_v\cap U_a$ is an analytic subset of $U_a$. Since the
		neighborhoods $U_a$ cover $S$, the subset $D_v$ is an analytic subset of $S$.
		Clearly, for each $v\in\Lambda$, the set $D_v$ is either $S$, empty, or  a  hypersurface of $S$.

		\begin{step}\label{step-moi-locus}
			The description of the Moishezon locus
		\end{step}
		
		We claim that
		\[
		\operatorname{Moi}(f)
		=
		\bigcup_{\substack{v\in\Lambda\\ q(v)>0}}D_v.
		\]
		Recall from Huybrechts' projectivity criterion
		(\cite[Theorem 2]{Huy01}, \cite[Theorem 3.11]{Huy99}) and Corollaries
		\ref{cor-hf-moishezon-4dim} and  
		\ref{cor-hf-moishezon} that, in both cases of the theorem, the following
		characterization holds:
		\begin{equation}\label{Moi-charac-both}
			\begin{aligned}
				X_t\text{ is Moishezon}
				\quad\Longleftrightarrow\quad
				&\exists\alpha\in H^2(X_t,\mathbb{Z})\cap H^{1,1}(X_t)
				\text{ such that }q_{X_t}(\alpha)>0.
			\end{aligned}
		\end{equation}
		
		Let $t\in\operatorname{Moi}(f)$. By \eqref{Moi-charac-both}, there exists
		\[
		\alpha\in H^2(X_t,\mathbb{Z})\cap H^{1,1}(X_t)
		\]
		such that $q_{X_t}(\alpha)>0$. Under the Gauss--Manin trivialization over
		$S$, the class $\alpha$ equals $v_t$ for some $v\in\Lambda$. By Claim
		\ref{BBF-posi-preserve}, one has $q(v)>0$, while
		$\alpha=v_t\in H^{1,1}(X_t)$ gives $t\in D_v$. Thus
		\begin{equation}\label{moi-subset-Dv}
			t\in
			\bigcup_{\substack{v\in\Lambda\\ q(v)>0}}D_v. 
		\end{equation}

		Conversely, suppose that $t\in D_v$ for some $v\in\Lambda$ with $q(v)>0$.
		Then
		\[
		v_t\in H^2(X_t,\mathbb{Z})\cap H^{1,1}(X_t),
		\]
		and Claim \ref{BBF-posi-preserve} gives $q_{X_t}(v_t)>0$. By the
		characterization \eqref{Moi-charac-both}, the fiber $X_t$ is Moishezon. Thus
		$t\in\operatorname{Moi}(f)$, and therefore
		\[
		\operatorname{Moi}(f)
		=
		\bigcup_{\substack{v\in\Lambda\\ q(v)>0}}D_v.
		\]
		
		Finally, $\Lambda=H^2(X_{s_0},\mathbb{Z})$ is countable, and each $D_v$ is 
		either $S$, empty, or  a hypersurface of $S$.
		Therefore $\operatorname{Moi}(f)$ is either $S$  or    at most (possibly empty)  a countable
		union of hypersurfaces of $S$.
		This completes the proof of Theorem
		\ref{thm-intro-hf-copy2}.
	\end{proof}

	\section*{Acknowledgements}
	
	The author would like to express his sincere gratitude to Professor Sheng Rao for  valuable discussions on related topics and constant encouragement and help.  He would also like to thank Professor I-Hsun Tsai for valuable
	discussions on related topics.


\begin{thebibliography}{GP}
		
		
		
		
		\bibitem[AHV18]{AHV18}
		J. M. Aroca, H. Hironaka, J. L. Vicente, \textit{Complex analytic desingularization},
		With a foreword by Bernard Teissier. Springer, Tokyo, 2018.
		
		\bibitem[ACRT18]{ACRT18}
		B. Anthes, A. Cattaneo, S. Rollenske, A. Tomassini,
		\textit{$\partial\bar\partial$-complex symplectic and Calabi--Yau manifolds:
			Albanese map, deformations and period maps},
		Ann. Global Anal. Geom. \textbf{54} (2018), 377-398.
		
		
		
		\bibitem[BL22]{BL22}
		B. Bakker, C. Lehn,
		\textit{The global moduli theory of symplectic varieties},
		J. Reine Angew. Math. \textbf{790} (2022), 223-265, \href{https://arxiv.org/pdf/1812.09748}{arXiv:1812.09748v4}.	
		
		
		\bibitem[Ba15]{Ba15}
		D. Barlet, \textit{Two semi-continuity results for the algebraic dimension of compact complex manifolds}, J. Math.
		Sci. Univ. Tokyo 22 (2015), no. 1, 39-54.
		
		
		
		\bibitem[Bea00]{Bea00}
		A.  Beauville,  \textit{Symplectic singularities},  Invent. Math., 139(3):541-549, 2000. 
		
		
		
		
		\bibitem[BCDG25]{BCDG25}
		I. Biswas, J. Cao, S. Dumitrescu, H. Guenancia,
		\textit{Geometry of $K$-trivial Moishezon manifolds: decomposition theorem and
			holomorphic geometric structures},
		Math. Ann. \textbf{391} (2025), no. 2, 3181--3220.
		
		
		
		
		
		\bibitem[CT18]{CT18}
		A. Cattaneo, A. Tomassini,
		\textit{Complex symplectic structures and the $\partial\bar{\partial}$-lemma},
		Annali di Matematica \textbf{197} (2018), 139-151.
		
		
		\bibitem[C25]{C25}
		J. Chen, \textit{Criteria for a fiberwise Fujiki/K\"ahler family to be locally Moishezon/projective},
		\href{https://arxiv.org/pdf/2503.07548v3}{arXiv:2503.07548v3}.
		
		
		\bibitem[CH24]{CH24}
		B. Claudon, A. H\"oring, \textit{Projectivity criteria for K\"ahler morphisms}, \href{https://arxiv.org/pdf/2404.13927}{arXiv:2404.13927}.	
		
		\bibitem[DHP24]{DHP24}
		O. Das, C. Hacon, M. P\u{a}un,
		\textit{On the $4$-dimensional minimal model program for K\"ahler varieties},
		Adv. Math. 443 (2024), Paper No. 109615.
		
		
		
		\bibitem[Fu78/79]{Fu78/79}
		A. Fujiki,  \textit{Closedness of the Douady spaces of compact K\"ahler spaces}, Publ. Res. Inst. Math. Sci. 14 (1978/79), no. 1, 1-52.
		
		
		
		\bibitem[Fj22]{Fj22}
		O. Fujino,
		\textit{Minimal model program for projective morphisms between complex analytic spaces},
		\href{https://arxiv.org/pdf/2201.11315}{arXiv:2201.11315}.
		
		
		\bibitem[GR84]{GR84}
		H. Grauert, R. Remmert, \textit{Coherent analytic sheaves}, Grundlehren der Mathematischen Wissenschaften [Fundamental Principles of Mathematical Sciences], 265. Springer-Verlag, Berlin, 1984.
		
		
		\bibitem[GLR13]{GLR13}
		D. Greb, C. Lehn, S. Rollenske,
		\textit{Lagrangian fibrations on hyperk\"ahler manifolds - On a question of Beauville},
		Ann. Sci. \'Ec. Norm. Sup\'er. (4) 46 (2013), no. 3, 375-403.
		
		
		
		\bibitem[GHJ02]{GHJ02}
		M. Gross, D. Huybrechts, D. Joyce,
		\textit{Calabi--Yau manifolds and related geometries:
			Lectures at a Summer School in Nordfjordeid, Norway, June 2001},
		Universitext, Springer-Verlag, Berlin, 2003.
		
		%
		
		\bibitem[Huy99]{Huy99}
		D. Huybrechts,
		\textit{Compact hyperk\"ahler manifolds: basic results},
		Invent. Math. \textbf{135} (1999), 63-113; Erratum, Invent. Math.
		\textbf{152} (2003), 209-212.
		
		
		\bibitem[Huy01]{Huy01}
		D. Huybrechts, \textit{Erratum to: Compact hyperk\"ahler manifolds: basic results}, arXiv:math/0106014; published as Invent. Math. 152(2003), 209-212.
		
		\bibitem[Huy16]{Huy16}
		D. Huybrechts, \textit{Lectures on K$3$ surfaces},
		Cambridge Studies in Advanced Mathematics, vol. 158,
		Cambridge University Press, 2016.
		
		
		\bibitem[Kl01]{Kl01}
		D. Kaledin, \textit{Symplectic resolutions: Deformations and birational maps}, \href{https://arxiv.org/pdf/math/0012008}{arXiv:math/0012008v2}.
		
		
		
		
		\bibitem[KS21]{KS21}
		S. Kebekus, C. Schnell,
		\textit{Extending holomorphic forms from the regular locus of a
			complex space to a resolution of singularities},
		J. Amer. Math. Soc. 34 (2021), no. 2, 315-368.
		
		\bibitem[KK13]{KK13}
		J. Koll\'ar,
		\textit{Singularities of the Minimal Model Program},
		with a collaboration of S. Kov\'acs,
		Cambridge Tracts in Mathematics, vol. 200,
		Cambridge University Press, Cambridge, 2013.
		
		
		\bibitem[Kol22]{Kol22}
		J.  Koll{\'a}r,    \textit{Moishezon morphisms}, Pure Appl. Math. Q. 18 (2022), no. 4, 1661-1687.
		
		\bibitem[KM98]{KM98}
		J. Koll\'ar, S. Mori,  \textit{Birational Geometry of Algebraic Varieties}, Cambridge Tracts in Mathematics,
		vol. 134, Cambridge University Press, Cambridge, 1998, With the collaboration of C. H. Clemens
		and A. Corti, Translated from the 1998 Japanese original.
		
		
		
		\bibitem[LRWW24]{LRWW24}
		M.  Li, S. Rao, K.  Wang, M.  Wang,  \textit{Smooth deformation limit of Moishezon manifolds is Moishezon},  \href{https://arxiv.org/pdf/2407.02022}{arXiv:2407.02022v2}.
		
		\bibitem[MM07]{MM07}	
		X. Ma, G. Marinescu, \textit{Holomorphic Morse inequalities and Bergman kernels}, Progress in Mathematics, 254.
		Birkh\"auser Verlag, Basel, 2007. 
		
		
		
		
		\bibitem[Mt02]{Mt02}
		K. Matsuki,  \textit{Introduction to the Mori program, Universitext}, Springer-Verlag, New York, 2002. xxiv+478 pp.
		
		\bibitem[Nam02]{Nam02}
		Y. Namikawa,
		\textit{Projectivity criterion of Moishezon spaces and density of projective
			symplectic varieties},
		Internat. J. Math. \textbf{13} (2002), no. 2, 125-135.
		
		\bibitem[OSS11]{OSS11}
		C. Okonek, M. Schneider, H. Spindler, \textit{Vector bundles on complex projective spaces, Corrected
			reprint of the 1988 edition, With an appendix by S. I. Gelfand}, Modern Birkh\"auser Classics.
		Birkh\"auser/Springer Basel AG, Basel, 2011.
		
		
		\bibitem[PR94]{PR94} Th. Peternell, R. Remmert, \textit{Differential calculus, holomorphic maps and linear structures on complex spaces}, Several complex variables VII, 99-143, Encyclopaedia of Mathematical Sciences volume 74, Springer-Verlag, Berlin 1994.
		
		
		
		\bibitem[RT21]{RT21}
		S. Rao, I-Hsun Tsai, \textit{Deformation limit and bimeromorphic embedding of Moishezon
			manifolds}, Commun. Contemp. Math. 23 (2021), no. 8, Paper No. 2050087, 50 pp.
		
		\bibitem[SSvL10]{SSvL10}
		M. Sch\"utt, T. Shioda, R. van Luijk,
		\textit{Lines on Fermat surfaces},
		J. Number Theory 130 (2010), no. 9, 1939-1963.
		
		\bibitem[St88]{St88}
		J. Stevens,  \textit{On canonical singularities as total spaces of deformations}, Abh. Math. Sem. Univ. Hamburg 58 
		(1988) 275-283.
		
		
		\bibitem[Tu11]{Tu11}
		L. W. Tu, \textit{An introduction to manifolds},
		Second Edition, Universitext, Springer, New York, 2011.
		
		\bibitem[Ue75]{Ue75}
		K. Ueno,
		\textit{Classification Theory of Algebraic Varieties and Compact Complex Spaces},
		Lecture Notes in Mathematics, vol. 439, Springer, Berlin--Heidelberg, 1975.
		
		\bibitem[Vr89]{Vr89}
		J. Varouchas,  \textit{K\"ahler spaces and proper open morphisms}, 	Math. Ann. 283 (1989), no. 1, 13-52.	
		
		\bibitem[Voi02]{Voi02}
		C. Voisin, \textit{Hodge Theory and Complex Algebraic Geometry I}, Cambridge Studies in Advanced Mathematics 76, Cambridge University Press, 2002.
		
		
		
		
		\bibitem[W21]{W21}
		J. Wang,
		\textit{On the Iitaka conjecture $C_{n,m}$ for K\"ahler fibre spaces},
		Ann. Fac. Sci. Toulouse Math. (6) 30 (2021), no. 4, 813-897.
		
		\bibitem[Wie00]{Wie00}
		J. Wierzba, \textit{Symplectic singularities, Ph.D. thesis}, Cambridge University, 2000.
		
		\bibitem[Wie03]{Wie03}
		J. Wierzba,
		\textit{Contractions of symplectic varieties},
		J. Algebraic Geom. 12 (2003), no. 3, 507-534.
		
		\bibitem[WW03]{WW03}
		J. Wierzba, J. A. Wi{\'s}niewski,
		\textit{Small contractions of symplectic $4$-folds},
		Duke Math. J. 120 (2003), no. 1, 65-95.
		
		\bibitem[Y01]{Y01}
		K.  Yoshioka,  \textit{Moduli spaces of stable sheaves on abelian surfaces},  Math. Ann. 321(4), 817-884 (2001).
		
		
	\end{thebibliography}
\end{document}